\documentclass[a4paper,10pt]{amsart}
\usepackage[latin1]{inputenc}
\usepackage[english]{babel}

\usepackage{a4wide}
\usepackage{mathrsfs}
\usepackage{amsfonts}
\usepackage{amsmath}
\usepackage{amssymb}
\usepackage{graphicx}
\usepackage{bbm,dsfont}
\usepackage{latexsym}
\usepackage{amsthm}
\usepackage{color}
\usepackage[parfill]{parskip}
\usepackage{eucal}
\usepackage[colorlinks=true,linkcolor=blue]{hyperref}
\numberwithin{equation}{section}

\advance\textheight by 1.5cm

\newcommand{\ud}{\,\mathrm{d}}
\newcommand{\1}{\mathbbm{1}}

\newcommand{\sX}{\mathcal{X}}
\newcommand{\x}{{\scriptstyle\sX}}
\newcommand{\sT}{\mathcal{T}}
\newcommand{\sS}{\mathcal{S}}
\renewcommand{\L}{\mathrm{L}}
\newcommand{\W}{\mathrm{W}}

\newcommand{\Schw}{\mathcal{S}}
\newcommand{\Xp}{\sX^p}

\newcommand{\WzC}{W^{2,p}_{0}\bigl([0,+\infty),U\bigr)}
\newcommand{\WzCY}{W^{2,p}_{0}\bigl([0,+\infty),Y\bigr)}

\newcommand{\A}{\mathcal{A}}
\newcommand{\sK}{\mathcal{K}}
\newcommand{\sP}{\mathcal{P}}
\newcommand{\sAl}{\mathcal{A}^{\lambda_0}}
\newcommand{\Al}{A-\lambda_0}
\newcommand{\sSt}{\mathcal{S}(t)}
\newcommand{\sStt}{\mathchoice{\bigl(\sSt\bigr)_{t\ge0}}{\bigl(\sSt\bigr)_{t\ge0}}{(\sSt)_{t\ge0}}{(\sSt)_{t\ge0}}}
\newcommand{\sTt}{\mathcal{T}(t)}
\newcommand{\sTtt}{\mathchoice{\bigl(\sTt\bigr)_{t\ge0}}{\bigl(\sTt\bigr)_{t\ge0}}{(\sTt)_{t\ge0}}{(\sTt)_{t\ge0}}}
\newcommand{\sA}{\mathcal{A}}
\newcommand{\sD}{\mathcal{D}}
\newcommand{\sM}{\mathcal{M}}
\newcommand{\B}{\mathcal{B}}
\newcommand{\sL}{\mathcal{L}}
\newcommand{\LT}{\sL}
\newcommand{\bF}{\mathbb{F}}
\newcommand{\bFt}{\tilde{\mathbb{F}}}

\newcommand{\E}{\mathcal{E}}
\newcommand{\Four}{\mathcal{F}}

\newcommand{\ad}{\bigl(A,D(A)\bigr)}

\newcommand{\duno}{\bigl(D_1,D(D_1)\bigr)}
\newcommand{\ddue}{\bigl(D_2,D(D_2)\bigr)}
\newcommand{\Tt}{\mathchoice{\bigl(T(t)\bigr)_{t\ge 0}}{\bigl(T(t)\bigr)_{t\ge 0}}{(T(t))_{t\ge 0}}{(T(t))_{t\ge 0}}}
\newcommand{\St}{(S(t))_{t\ge 0}}
\newcommand{\Sunot}{\mathchoice{\bigl(S_1(t)\bigr)_{t\ge 0}}{\bigl(S_1(t)\bigr)_{t\ge 0}}{(S_1(t))_{t\ge 0}}{(S_1(t))_{t\ge 0}}}
\newcommand{\Sduet}{\mathchoice{\bigl(S_2(t)\bigr)_{t\ge 0}}{\bigl(S_2(t)\bigr)_{t\ge 0}}{(S_2(t))_{t\ge 0}}{(S_2(t))_{t\ge 0}}}

\newcommand{\LpRmY}{\mathchoice{\L^p\bigl((-\infty,0],Y\bigr)}{\L^p\bigl((-\infty,0],Y\bigr)}{\L^p((-\infty,0],Y)}{\L^p((-\infty,0],Y}}
\newcommand{\LpRmX}{\mathchoice{\L^p\bigl((-\infty,0],X\bigr)}{\L^p\bigl((-\infty,0],X\bigr)}{\L^p((-\infty,0],X)}{\L^p((-\infty,0],X}}
\newcommand{\LpRpX}{\mathchoice{\L^p\bigl([0,+\infty),X\bigr)}{\L^p\bigl([0,+\infty),X\bigr)}{\L^p([0,+\infty),X)}{\L^p([0,+\infty),X}}
\newcommand{\LeRpX}{\mathchoice{\L^1\bigl([0,+\infty),X\bigr)}{\L^1\bigl([0,+\infty),X\bigr)}{\L^1([0,+\infty),X)}{\L^1([0,+\infty),X}}
\newcommand{\LeRp}{{\L^1[0,+\infty)}}

\newcommand{\LpRpU}{\mathchoice{\L^p\bigl([0,\infty),U\bigr)}{\L^p\bigl([0,\infty),U\bigr)}{\L^p([0,\infty),U)}{\L^p([0,\infty),U)}}
\newcommand{\LeRpU}{\mathchoice{\L^1\bigl([0,\infty),U\bigr)}{\L^1\bigl([0,\infty),U\bigr)}{\L^1([0,\infty),U)}{\L^1([0,\infty),U)}}
\newcommand{\LpRpY}{\mathchoice{\L^p\bigl([0,\infty),Y\bigr)}{\L^p\bigl([0,\infty),Y\bigr)}{\L^p([0,\infty),Y)}{\L^p([0,\infty),Y)}}
\newcommand{\WpY}{\W^{1,p}\bigl((-\infty,0],Y\bigr)}
\newcommand{\WpzY}{\W^{1,p}_0\bigl((-\infty,0],Y\bigr)}
\newcommand{\WpU}{\W^{1,p}\bigl([0,\infty),U\bigr)}
\newcommand{\WpX}{\W^{1,p}\bigl([0,\infty),X\bigr)}

\newcommand{\WzpU}{\W^{2,p}\bigl([0,\infty),U\bigr)}
\newcommand{\el}{\varepsilon_{\lambda}}

\newcommand{\rhoA}{\rho(A)}

\newcommand{\N}{\mathbb{N}}
\newcommand{\R}{\mathbb{R}}
\newcommand{\C}{\mathbb{C}}

\newcommand{\den}{w_n}
\newcommand{\Bt}{\mathcal{B}(t)}
\newcommand{\Btt}{\mathchoice{\bigl(\Bt\bigr)_{t\ge0}}{\bigl(\Bt\bigr)_{t\ge0}}{(\Bt)_{t\ge0}}{(\Bt)_{t\ge0}}}

\newcommand{\Ct}{\mathcal{C}(t)}
\newcommand{\Ctt}{\mathchoice{\bigl(\Ct\bigr)_{t\ge0}}{\bigl(\Ct\bigr)_{t\ge0}}{(\Ct)_{t\ge0}}{(\Ct)_{t\ge0}}}
\newcommand{\Ft}{\mathbb{F}(t)}
\newcommand{\Ftt}{\mathchoice{\bigl(\Ft\bigr)_{t\ge0}}{\bigl(\Ft\bigr)_{t\ge0}}{(\Ft)_{t\ge0}}{(\Ft)_{t\ge0}}}

\newcommand{\RlA}{R(\lambda,A)}

\newcommand{\RlsA}{R(\lambda,\A)}
\newcommand{\RlAme}{R(\lambda,A_{-1})}
\newcommand{\inc}{\overset{\text{c}\;}{\hookrightarrow}}

\newcommand{\rg}{\operatorname{rg}}
\renewcommand{\Re}{\operatorname{Re}}

\newcommand{\spn}{\operatorname{span}}
\newcommand{\cd}{{\raisebox{1.3pt}{{$\kern0.5pt\scriptscriptstyle\bullet\kern0.5pt$}}}}

\newtheorem{definition}{Definition}[section]
\newtheorem{lemma}[definition]{Lemma}
\newtheorem{proposition}[definition]{Proposition}
\newtheorem{theorem}[definition]{Theorem}
\newtheorem{corollary}[definition]{Corollary}

\theoremstyle{remark}
\newtheorem{remark}[definition]{Remark}
\newtheorem{assumptions}[definition]{Assumption}

\allowdisplaybreaks[1]

\title[A Semigroup Characterization of Well-Posed Linear Control Systems]{A Semigroup Characterization of\\ Well-Posed Linear Control Systems}

\author{M. Bombieri, K.-J. Engel}
\date{\today}

\begin{document}
\subjclass[2010]{93C05, 47D03, 35K05}
\keywords{Well-posed linear systems, admissible control operator, admissible output operator, Laplace transform, Fourier multiplier, Lax--Phillips semigroup.}
\thanks{The authors would like to thank A.~B\'atkai for many valuable discussions.}

\begin{abstract}
We study the well-posedness of a linear control system \ref{csu} with unbounded control and observation operators. To this end we associate to our system an operator matrix $\sA$ on a product space $\Xp$ and call it $p$-well-posed if $\sA$ generates a strongly continuous semigroup on $\Xp$.
Our approach is based on the Laplace transform and Fourier multipliers. The results generalize and complement those of \cite{CW:89}, \cite{tesi2} and are illustrated by a heat equation with boundary control and point observation.
\end{abstract}
\maketitle


\section{Introduction}
In this paper we investigate the well-posedness of linear control systems of the form
\begin{equation*}
\tag*{$\Sigma(A,B,C,D)$}
\label{csu}
\begin{cases}
\dot{x}(t)=Ax(t)+Bu(t),& t\ge0, \\
y(t)=Cx(t)+Du(t),& t\ge0, \\
x(0)=x_0.
\end{cases}
\end{equation*}
The operators $A, B,C ,D$ are linear and defined on Banach spaces $X, Y$ and $U$, called \emph{state}-, \emph{observation}- and \emph{control space}, respectively, and satisfy the following hypotheses:\footnote{For the definition of the inter- and extrapolation spaces $X_1$, $X_{-1}$ see below.}
\begin{itemize}
	\item $A:D(A)\subset X \to X$, called the \emph{state operator}, is the generator of a $C_0$-semigroup,
	\item $B \in \sL(U,X_{-1})$ is the \emph{control operator},
	\item $C \in \sL(X_1, Y)$ is the \emph{observation operator},
	\item $D \in \sL(U,Y)$ is the \emph{feedthrough operator}.
\end{itemize}
For the motivation, concrete examples and a systematic treatment of such systems we refer to \cite{CZ95}, \cite{Hadd1}, \cite{Hadd2}, \cite{nuovoweiss}, \cite{Weiss} and the references therein. Moreover, in Section~\ref{example} we illustrate our results considering a heat equation with boundary control and point observation.

Generalizing an idea of Grabowski and Callier \cite{GC}, see also Engel \cite{adm},  we associate to our system an operator matrix $\bigl(\A, D(\A)\bigr)$ defined on an appropriate product space $\Xp$ depending on $p\ge1$. We then call \ref{csu} \emph{$p$-well posed} if this operator matrix generates a $C_0$-semigroup $\sT=\sTtt$ on $\Xp$.

In other words,  \ref{csu} is well-posed if the Cauchy problem
\begin{equation}\label{ACP-A}
\begin{cases}
\dot{\x}(t) = \A \x(t), & t\ge 0,\\
\x(0) = \x_0 
\end{cases}
\end{equation}
is well-posed on $\Xp$ in the sense of Hadamard (see \cite[Sect.~II.6]{EN}).

It turns out that this definition of well-posedness leads to the concept of $p$-admissibility of the control operator $B$ and the observation operator $C$ as studied, e.g., by Staffans and Weiss, see \cite{lebext}, \cite{W-Ad-contr}, \cite{Weiss}, \cite{tesi1}, \cite{tesi2}.

We mention that the semigroup $\sT$ generated by $\A$ already appears in \cite{tesi1}, \cite{tesi2}, \cite{Sta:05} where it is called the ``\emph{Lax-Phillips semigroup}".

To carry out the program sketched above we start from the generator $\ad$ of a semigroup $\Tt$ on a Banach space $(X,\|\cd\|)$. 
We then consider the associated abstract Sobolev spaces (see \cite[Sect.~II.5]{EN}) defined by
\begin{itemize}
	\item $X_1 := \bigl(D(A), \|\cd\|_A\bigr)$, where $\|\cd\|_A$ is the graph norm given by $\|x\|_A:=\|x\|+\|Ax\|$,
	\item $X_{-1} := \bigl(X, \|\cd\|_{-1}\bigr)\;\widetilde{}$, where $\|x\|_{-1}:= \|\RlA x\|$ for $x\in X$ and some fixed $\lambda\in\rhoA$.
\end{itemize}
Then $\Tt$ uniquely extends to the extrapolated semigroup $\bigl(T_{-1} (t)\bigr)_{t \ge 0}\subset\sL(X_{-1})$ with generator $\bigl(A_{-1}, D(A_{-1})\bigr)$ where $D(A_{-1})=X$.

For the observation operator $C$ we define as in \cite[Sect.~4]{lebext} its \emph{Lebesgue extension} $C_{L} : D ( C_L ) \subset X \to Y$ by
\begin{align*}
D (C_L) &:= \biggl\{x \in X :\lim\limits_{t\searrow0}C\,\frac{1}{t}\int_0^t T(s)x\ud s\text{ exists}\biggr\}, \\
C_L x &:=  \lim\limits_{t\searrow0}C\,\frac{1}{t}\int_0^t T(s)x\ud s\quad \text{for all }x \in D (C_L).
\end{align*}

Now the following holds, see \cite[Prop. 4.3]{lebext}.

\begin{proposition}\label{D(C_L)-Ban}
The space $D(C_L)$ endowed with the norm
\[
\|x\|_{L}:= \|x\|_X+\sup_{t\in (0,1]}\biggl\|C\frac{1}{t}\int_0^t T(s)x \ud s\biggr\|_X, \quad x\in D(C_L)
\]
is a Banach space. Moreover, the embeddings $X_1\inc D(C_L)\inc X$ are continuous and $C_L\in\sL\bigl(D(C_L),Y\bigr)$.
\end{proposition}

To proceed we need the following stability and compatibility conditions. The latter relates the operators $A$, $B$ and $C$,  cf. \cite[Sect.~II.A]{Hel:76}. For more information and several equivalent conditions see \cite[Thm.5.8]{tesi1}.

\begin{assumptions}\label{assum-reg}
If not stated otherwise, in the sequel we always make the following hypotheses.
\begin{enumerate}
\item[(i)] The semigroup $\Tt$ is uniformly exponentially stable, i.e., there exist $K\ge1$ and $\omega<0$ such that
\begin{equation}\label{eq:T(t)-stab}
\bigl\|T(t)\bigr\|\le Ke^{\omega t}\quad\text{for all }t\ge0.
\end{equation}
\item[(ii)] The system \ref{csu} is \emph{compatible} (or \emph{regular}), i.e., for some $\lambda\in\rho(A)$ we have
\begin{equation}\label{bild}
\rg\bigl(\RlAme B\bigr)\subset D(C_L).
\end{equation}
\end{enumerate}
\end{assumptions}

While assumption~(i) is made only for convenience and to simplify the presentation (cf. also Remark~\ref{rem:not-stable}), assumption~(ii) is essential and cannot be omitted.
Note that if the inclusion \eqref{bild} holds for some $\lambda\in\rho(A)$ then by the resolvent equation it holds for all $\lambda\in\rho(A)$. Moreover, the closed graph theorem and Proposition~\ref{D(C_L)-Ban} then imply that
\begin{equation}\label{eq:Delta-bdd}
 C_L R(\lambda, A_{-1}) B \in \sL(U, Y)\quad\text{for all }\lambda \in \rho(A).
\end{equation}

We close this introduction with a brief outline of this work. In Section~\ref{admoo} we introduce the operator matrix $\sA$ from \eqref{ACP-A} on the space $\Xp$ and compute its resolvent $\RlsA$. In Section~\ref{sec:char-well} we show how the concept of admissibility for the observation operator $C$, the control operator $B$ and the pair $(B,C)$  is related to the existence of strongly continuous operator families having as Laplace transforms the entries of $\RlsA$. Section~\ref{sec:FM} is dedicated to the characterization of admissible pairs in terms of a resolvent condition which leads to so-called Fourier multipliers. In Section~\ref{A-csu} we summarize our results from the previous section and give several characterizations of the generator property of $\sA$, i.e., of the well-posedness of \ref{csu}. In the final Section~\ref{example} we illustrate our results and show the well-posedness of a controlled heat equation.


\section{The Operator Matrix $\A$}
\label{admoo}

In this section we define the operator matrix $\A$ appearing in \eqref{ACP-A} which governs the control system $\Sigma(A,B,C,D)$. To this end we first fix some $1\leq p < \infty$. Then we introduce
\begin{itemize}
\item the space $E_1^p:=\LpRmY$ of \emph{possible observations},
\item the space $E_2^p := \LpRpU$ of \emph{possible controls}, and
\item the \emph{extended state space} $\Xp=E_1^p\times X\times E_2^p$.
\end{itemize}
On $\Xp$ (equipped with an arbitrary product norm) we define the operator matrix
\begin{align}\label{matA}
\A:&=
\begin{pmatrix}
\frac{\ud}{\ud s} & 0 & 0 \\
0 & A_{-1}   & B\delta_0 \\
0 & 0   & \frac{\ud}{\ud s}
\end{pmatrix},\\
\label{matA-D}
D(\A):&=\left\{\left(
\begin{smallmatrix}
y\\
x\\
u
\end{smallmatrix}\right)
\in\E: A_{-1}x+Bu(0)\in X,\; y(0)=C_Lx+Du(0)
\right\},
\end{align}
where $\delta_0: \WpU\subset E_2^p\to U$ denotes the point evaluation given by $\delta_0u:=u(0)$ and
\[
\E:= \WpY\times D(C_L)\times  \WpU.
\]
Note that there is a close relation between this operator matrix and the system \ref{csu}. In fact, on the second row of \eqref{matA} we can recognize the first equation of the system \ref{csu}, while in the definition \eqref{matA-D} of the domain of $\A$ the output equation of \ref{csu} appears as a boundary condition. In Section~\ref{A-csu} we will return to the relation between the matrix $\A$ and the system \ref{csu}.

As already mentioned in the introduction we define the well-posedness of \ref{csu} in terms of the operator matrix $\A$.

\begin{definition}
The system \ref{csu} is called \emph{$p$-well-posed} if the operator matrix $\A$ in \eqref{matA}, \eqref{matA-D} generates a $C_0$-semigroup on $\Xp$.
\end{definition}

In order to characterize the generator property of $\A$ in terms of its entries, we follow ideas developed in \cite{ST} for $2\times 2$-matrices. To do so we introduce some more notation.

First we consider the operators
\begin{itemize}
\item $D_1:=\frac{\ud}{\ud s}:D(D_1)\subset E_1^p\to E_1^p$ with domain $D(D_1):=\WpzY:=\bigl\{y\in \WpY:y(0)=0\bigr\}$,
\item $D_2 := \frac{\ud}{\ud s}:D(D_2)\subset E_2^p\to E_2^p$ with domain $D(D_2) := \WpU$.
\end{itemize}
Note that $\duno$ generates the left shift semigroup
$\sS_1=\Sunot$ on $E_1^p$ given by
\[
\bigl(S_1(t)y\bigr)(s):=
\begin{cases}
y(t+s)&\text{if }t+s\le0,\\
0&\text{if }t+s>0,
\end{cases}
\]
while $\ddue$ generates the left shift semigroup $\sS_2=\Sduet$ on $E_2^p$, see \cite[Sect.~II.2.b]{EN} for more details.

Next, for $\lambda \in \C$ with $\Re\lambda>0$ we consider $\el\in\L^p\bigl(-\infty,0]$ defined by $\el(s):=e^{\lambda s}$. Then for an operator $Q:D(Q)\subset V\to Y$ (where $V=X$ or $V=E_2^p$) we define
\[
\el\otimes Q:D(Q)\subset V\to\LpRmY,\quad
\bigl((\el\otimes Q)x\bigr)(s):=\el(s)\cdot Qx=e^{\lambda s}\cdot Qx.
\]

We are now able to represent the matrix $\lambda-\A$ as follows.

\begin{proposition}
Let $\Re\lambda>0$. Then
\begin{align}\label{factorize}
\lambda-\A&=
\begin{pmatrix}
\lambda-D_1&0&0\\
0&\lambda-A&0\\
0&0&\lambda-D_2
\end{pmatrix}
\cdot
\begin{pmatrix}
Id&-\el\otimes C_L&-\el\otimes D\delta_0\\
0&Id&-\RlAme B\delta_0\\
0&0&Id
\end{pmatrix}
\\\notag
&=:\A_\lambda\cdot(Id-\sK),
\end{align}
where $D(\A_\lambda) = D(D_1) \times D(A) \times D(D_2)$ and $D(\sK)=E_1^p\times D(C_L)\times  \WpU$.
\end{proposition}

\begin{proof}
A simple computation shows that $D\bigl(\A_\lambda(Id-\sK)\bigr):=\bigl\{\x\in D(\sK):(Id-\sK)\x\in D(\A_\lambda)\bigr\}$ coincides with $D(\A)$ and that $(\lambda-\A)\x=\A_\lambda(Id-\sK)\x$ for all $\x\in D(\A)$.
\end{proof}

Using the above representation of $\lambda-\A$ it is easy to find an explicit representation for the resolvent $\RlsA$ of $\A$. To this end we denote by $\LT$ the Laplace transform, i.e., for $\Re\lambda>0$ and $u\in E_2^p=\LpRpU$ we define
\[
(\LT u)(\lambda):=\LT_{\lambda} u:=\hat u(\lambda):=\int_0^{+\infty}e^{-\lambda r}u(r)\ud r.
\]

\begin{corollary}\label{res-sA} For $\lambda \in \C$ with $\Re\lambda>0$ we have $\lambda\in\rho(\A)$ and
\begin{equation}\label{resolv}
R(\lambda,\A)=
\begin{pmatrix}
R(\lambda,D_1)&\el\otimes C\RlA&\el\otimes C_L\RlAme B\LT_{\lambda}+\el\otimes D\LT_{\lambda}\\
0&\RlA&\RlAme B\LT_{\lambda}\\
0&0&R(\lambda,D_2)
\end{pmatrix}.
\end{equation}
\end{corollary}

\begin{proof} Note that $\Re\lambda>0$ implies $\lambda\in\rho(D_1)\cap\rho(A)\cap\rho(D_2)$.
Using \eqref{factorize} we then obtain
\begin{equation*}
R(\lambda,\A)=
\begin{pmatrix}
Id & \el\otimes C_L& \el\otimes C_L \RlAme B\delta_0 + \el\otimes D\delta_0\\
0 & Id & \RlAme B\delta_0\\
0 & 0 & Id
\end{pmatrix}
\cdot
\begin{pmatrix}
R(\lambda,D_1) & 0 & 0\\
0 & R(\lambda,A) & 0\\
0 & 0 & R(\lambda,D_2)
\end{pmatrix}.
\end{equation*}
Since
\begin{align*}
\delta_0R(\lambda,D_2)u
= \delta_0 \Bigl(e^{\lambda \cd}\int_\cd^{\infty}e^{-\lambda r}u(r)\ud r\Bigr) 
=\LT_{\lambda} u \qquad \text{for all }u \in \LpRpU
\end{align*}
this implies \eqref{resolv}.
\end{proof}

\section{Characterization of Admissibility in the Time Domain}
\label{sec:char-well}

In this section we study the possible entries of a semigroup generated by the operator matrix $\A$. As we will see this leads to the concept of admissibility for the observation operator $C$, the control operator $B$ and the pair $(B,C)$. Our approach is based on the Laplace transform which relates a semigroup to the resolvent of its generator. More precisely, we use the following result, see \cite[Thm.~3.1.7]{AB}.

\begin{lemma}\label{lem-char-ABHN} Let $\sStt\subset\sL(\sX)$ be an exponentially bounded and strongly continuous operator family on a Banach space $\sX$. Then the following assertions are equivalent.
\begin{enumerate}
\item[(a)] There exists an operator $\sD:D(\sD)\subset\sX\to\sX$ and some $\lambda_0\in\R$ such that $(\lambda_0,+\infty)\subset\rho(\sD)$ and
\[
\LT\bigl(\sS(\cd)\x\bigr)(\lambda)=R(\lambda,\sD)\x\quad
\text{for all }\lambda>\lambda_0\text{ and all }\x\in\sX.
\]
\item[(b)] $\sStt$ is a $C_0$-semigroup.
\end{enumerate}
Moreover, in this case $\sD$ coincides with the generator of $\sStt$.
\end{lemma}

Recall that in Corollary~\ref{res-sA} we already computed the resolvent of $\A$. The idea is now to define (at least on dense subspaces) operator families $\bigl(T_{jk}(t)\bigr)_{t\ge0}$ for $j,k=1,2,3$ such that their Laplace transforms coincide with\footnote{Here $[\sM]_{jk}$ indicates the $jk$-th entry $m_{jk}$ of the matrix $\sM=(m_{jk})_{3\times 3}$.} $\bigl[\RlsA\bigr]_{jk}$ (on these subspaces). Hence, if $\sA$ is the generator of a $C_0$-semigroup $\sTtt$ these operator families $\bigl(T_{jk}(t)\bigr)_{t\ge0}$ must have (by denseness unique) bounded, strongly continuous extensions. Indeed, by the uniqueness theorem for the Laplace transform (see \cite[Thm.~1.7.3]{AB}), they are the only possible entries of $\sTt$. On the other hand, if these operator families $\bigl(T_{jk}(t)\bigr)_{t\ge0}$ have bounded, strongly continuous extensions, then their Laplace transforms give the entries of the resolvent of $\RlsA$, hence by Lemma~\ref{lem-char-ABHN} the matrix $\sA$ is a generator.

This idea works without problems for all entries of $\RlsA$ and $\sTtt$ below and on the diagonal. More precisely if $\A$ is a generator then the generated semigroup has necessarily the form
\[
\sTt=
\begin{pmatrix}
S_1(t)&*&*\\
0&T(t)&*\\
0&0&S_2(t)
\end{pmatrix}.
\]
Therefore, we only have to consider the remaining three entries. This will be done in the following subsections.

\subsection{The Entry \boldmath{$T_{12}(t)$} and Admissible Observation Operators}
\label{T12}

For $t\ge 0$ we define the operators
\begin{align*}
&T_{12}(t):D(A)\subset X\to E_1^p,\\
\bigl(&T_{12}(t)x\bigr)(s):=\1_{[-t,0]} (s)CT(t+s)x,\quad s\in\R_-.
\end{align*}

We first verify some basic properties of this operator family.

\begin{lemma}\label{lem:t_12-welldef}
For every $x\in D(A)$ the function $T_{12}(\cd)x:\R_+\to E_1^p$ is well-defined, continuous and bounded.
\end{lemma}

\begin{proof} Since $x\in D(A)$ we can write
\[
\bigl(T_{12}(t)x\bigr)(s)=CA^{-1}\cdot \1_{[-t,0]} (s) T(t+s) Ax,
\]
where $CA^{-1}\in\sL(X,Y)$.
Hence to prove the claims it suffices to consider the simpler function $g:\R_+\to \LpRmX$ defined by $g(t):=\1_{[-t,0]} (\cd) T(t+\cd) z$ where $z:=Ax\in X$. By assumption $\Tt$ is exponentially stable, thus we get
\begin{align*}
\bigl\|g(t)\bigr\|_{\LpRmX}^p
&=\int_{-t}^{0}\bigl\|T(t+s)z\bigr\|_X^p\ud s\\
&\le\tfrac{K}{\omega p}\bigl(e^{\omega pt}-1\bigr)\cdot\|z\|_X^p\\
&\le\tfrac{-K}{\omega p}\cdot\|z\|_X^p\qquad\text{for all }t\ge0.
\end{align*}
This proves that $g$ is well-defined and bounded. To show its continuity let $0\le r\le t$. Then
\begin{align*}
\|g(t)-g(r)\|_{\LpRmX}&=\bigl\|\1_{[-t,0]} (\cd)T(t+\cd)z-\1_{[-r,0]} (\cd)T(r+\cd)z\bigr\|_{\LpRmX}\\
&\le\biggl(\int_{-t}^{-r}\bigl\|T(t+s)z\bigr\|_X^p\ud s\biggr)^{\frac1p}
+\biggl(\int_{-r}^{0}\bigl\|\bigl(T(t+s)-T(r+s)\bigr)z\bigr\|_X^p\ud s\biggr)^{\frac1p}
\\
&\le (t-r)^{\frac1p}K\cdot\|z\|_X+r^{\frac1p}K\bigl\|\bigl(T(t-r)-Id\bigr)z\bigr\|_X
\to0\quad\text{as $t-r\to0$},
\end{align*}
where again we used Assumption~\ref{assum-reg}.(i).
\end{proof}

By the previous result we can Laplace transform $T_{12}(\cd)$.

\begin{lemma}\label{laplC}
For every $x\in D(A)$ and $\lambda \in \C$ with $\Re \lambda > 0$ we have
\[
\LT\bigl(T_{12}(\cd)x\bigr)(\lambda)=\el\otimes C\RlA x=[\RlsA]_{12}x.
\]
\end{lemma}

\begin{proof}
For $x \in D(A), \lambda \in \C$ with $\Re \lambda > 0$ and $s \in (-\infty, 0]$ we obtain
\begin{align*}
\Bigl(\LT\bigl(T_{12}(\cd)x\bigr)(\lambda)\Bigr)(s) & = C A^{-1} \int_0^{\infty} e^{-\lambda t} \1_{[-t, 0]}(s) T (t+ s) A x \ud t \\
&= C A^{-1} \int_{-s}^{\infty} e^{-\lambda t} T (t+ s) A x \ud t \\
&= e^{\lambda s}C A^{-1} \int_{0}^{\infty} e^{-\lambda t} T (t) A x \ud t  \\
&= \bigl(\el \otimes C \RlA x\bigr)(s).\qedhere
\end{align*}
\end{proof}

We proceed by introducing the following well-known notion from control theory (see, e.g., \cite{lebext}) which is closely related to the entry $T_{12}(t)$.

\begin{definition}
The observation operator $C\in\sL(X_1,Y)$ is called \emph{$p$-admissible} (with respect to $A$) if there exists $t_0>0$ and a constant $M\ge 0$ such that
\begin{equation*} 
	\int_0^{t_0}\bigl\|CT(s)x\bigr\|_{Y}^{p}\ud s \leq M\|x\|_X^p\qquad\text{for all }x\in D(A).
\end{equation*}
\end{definition}

\begin{remark} \label{rem:admiss-T_12}
Since for $t\ge0$ and $x\in D(A)$ we have
\begin{align*}
\|T_{12}(t)x\|_{E_1^p}^p
=\int_{-t}^{0}\bigl\|CT(t+s)x\bigr\|_Y^p\ud s
=\int_{0}^{t}\bigl\|CT(s)x\bigr\|_Y^p\ud s
\end{align*}
the observation operator $C\in\sL(X_1,Y)$ is $p$-admissible if and only if $T_{12}(t):D(A)\subset X\to E_1^p$ has a bounded extension in $\sL(X,E_1^p)$ for some $t>0$. Moreover we note that for $C\in\sL(X_1,Y)$ the condition to be a $p$-admissible observation operator gets stronger with growing $p\ge1$.
\end{remark}

Next we give different characterizations of admissibility for observation operators where we have to distinguish the cases $p>1$ and $p=1$.

\begin{lemma}\label{carammC}Let $p>1$. Then the operator $C$ is $p$-admissible if and only if for every $x \in X$ we have $T(\cd) x \in \L^p\bigl([0,t_0],D(C_L)\bigr)$.
\end{lemma}

\begin{proof}
We first introduce the following operators and spaces referring to the setting of Lemma~\ref{lem-app}
\begin{align*}
&\tilde{Q}:X\to\L^p\bigl([0,t_0],X\bigr),&&\kern-50pt\tilde{Q} x := T(\cd) x,\\
&Q:D(A)\subset X\to\L^p\bigl([0,t_0],D(C_L)\bigr),&&\kern-50pt Qx:=T(\cd)x,
\end{align*}
$D:=D(A)$, $V:=X$, $W:=\L^p\bigl([0,t_0],D(C_L)\bigr)$ and $Z:=\L^p\bigl([0,t_0],X\bigr)$. Here by Proposition~\ref{D(C_L)-Ban} $D(C_L)$  is a Banach space and for $x\in D(A)$ we have $Qx \in C\bigl([0,t_0], X_1\bigr)\subset\L^p\bigl([0,t_0],D(C_L)\bigr)$.

We now show that if $C$ is $p$-admissible, then there exists a constant $\bar M \ge 0$ such that
\begin{equation}\label{condizionea}
\left(\int_0^{t_0} \| T(s) x\|_L^p\ud s \right)^{\frac{1}{p}} \leq \bar M \| x\|_X\qquad\text{for all } x \in D(A).
\end{equation}
To do so we recall that for a function $f \in\L^1_{\textrm{loc}}(\R)$ its \emph{Maximal Function} (cf. \cite[Sect.I.1]{stein}) is defined by
\[
(\mathcal{M}f)(s) := \sup_{t > 0} \frac{1}{2t} \int_{s-t}^{t+s} |f(r)|\ud r.
\]
Then the Hardy--Littlewood Maximal Theorem (see \cite[Thm.I.1.1]{stein}) asserts that $\mathcal{M}f \in\L^p(\R)$ for $f \in\L^p(\R)$ with $1 < p \leq \infty$  and that there exists a constant $C_p$ depending only on $p$ such that
\[
\left\|\mathcal{M}f \right\|_{\L^p(\R)} \leq C_p \left\| f \right\|_{\L^p(\R)}.
\]
Using this for
\[
f(r):=
\begin{cases}
\| CT(r)x \|_Y &\text{if } 0 \leq r \leq t_0,\\
\qquad 0 & \text{else},
\end{cases}
\]
where $x \in D(A)$, we obtain
\begin{align*}
\left(\int_0^{t_0} \| T(s) x\|_L^p \ud s\right)^{\frac{1}{p}} 
& \leq \left(\int_0^{t_0} \| T(s) x\|_X^p \ud s \right)^{\frac{1}{p}} +  \left(\int_0^{t_0} \sup_{t \in (0,1]} \biggl\| \frac{1}{t} \int_0^t C T(r) T(s) x \ud r\biggr\|_Y^p \ud s \right)^{\frac{1}{p}} \\ 
 &\leq Kt_0^{\frac1p} \| x \|_X + \left(\int_0^{t_0} \biggl(\sup_{t \in (0,1]} \frac{1}{t} \int_s^{t+s}  \bigl\|  C T(r)  x \bigr\|_Y\ud r\biggr)^p \ud s\right)^{\frac{1}{p}} \\ \nonumber
& \leq  Kt_0^{\frac1p} \| x \|_X  + 2\left\| \mathcal{M}f \right\|_{\L^p(\R)} \\ 
& \leq  Kt_0^{\frac1p} \| x \|_X  + 2C_p M^{\frac{1}{p}} \| x \|_X.
\end{align*}
Here we used that the semigroup $\Tt$ is bounded by a constant $K \ge 1$, the Hardy-Littlewood Maximal Theorem and the fact that the observation operator $C$ is $p$-admissible.
This shows \eqref{condizionea} for $\bar M :=Kt_0^{\frac1p} + 2C_p M^{\frac{1}{p}}$.

It thus follows that if $C$ is $p$-admissible, then condition~(a) of Lemma~\ref{lem-app} is satisfied and we conclude that $\tilde Q \in \sL\bigl(X, \L^p\bigl([0,t_0],D(C_L)\bigr)\bigr)$, i.e., $T(\cd) x \in\L^p\bigl([0,t_0],D(C_L)\bigr)$ for every $x \in X$.
In particular, for every $x \in X$ this implies $T(r) x \in D(C_L)$ for almost all $r \in [0, t_0]$.

Conversely, if $T(\cd) x \in \L^p\bigl([0,t_0],D(C_L)\bigr)$ for every $x \in X$ then $ \rg(\tilde Q) \in \L^p\bigl([0,t_0],D(C_L)\bigr)$.
Thus condition~(b) of Lemma~\ref{lem-app} is satisfied and we conclude $\tilde Q \in \sL\bigl(X, \L^p\bigl([0,t_0],D(C_L)\bigr)\bigr)$.
From Proposition~\ref{D(C_L)-Ban} it then follows $C_L \tilde Q \in \sL\bigl(X, \L^p\bigl([0,t_0],Y\bigr)\bigr)$ and therefore $C$ is $p$-admissible.
\end{proof}

\begin{remark}
If $C$ is $p$-admissible for some $p >1$ then the previous result together with the semigroup property imply that  for all $x \in X$ we have $\rg\bigl(T(t)x\bigr)\subset D(C_L)$ for almost all $t\ge0$.
\end{remark}

As we will see next the range condition in the previous remark holds also in the case $p=1$ (see also \cite[Theorem 4.5]{lebext}).

\begin{lemma}\label{caso:p=1}
If $C$ is $1$-admissible and $x \in X$, then $T(t)x \in D(C_L)$ for almost all $t \ge 0$.
\end{lemma}

\begin{proof}
If $C$ is $1$-admissible, then the map $Q:D(A) \to \L^1\bigl([0,t_0],Y\bigr)$ given by $Qx =C T(\cd) x$ has a bounded continuous extension $\bar{Q}$ on all of $X$.
Furthermore for all $x \in D(A)$, $t\ge 0$ and  $r >0$ we have
\[
\frac{1}{r} \int _t ^{t+r} \bigl(\bar{Q}x\bigr)(s) \ud s = C \frac{1}{r} \int_0^{r} T\left(t+s\right) x \ud s.
\]
Since both sides depend continuously on $x$, the equality holds for every $x\in X$.
Letting $r\to0$, it follows that $T(t)x \in D\left(C_L\right)$ if and only if $\bar{Q}x$ has a Lebesgue point in $t$.
Hence by the Lebesgue differentiation theorem (see \cite[Thm.II.2.9]{DU:77}) it follows that $T(t)x \in D\left(C_L\right)$ for almost all $t\ge 0$.
\end{proof}

Finally we prove the following result which is closely related to \cite[Prop.~2.3]{lebext}. Here we need again Assumption~\ref{assum-reg}.(i).

\begin{lemma}\label{Ctuttit}
If the observation operator $C$ is $p$-admissible, then there exists $M_C\ge0$ such that
\begin{equation}\label{eq:Ctuttit}
	\int_0^{t}\bigl\|CT(s)x\bigr\|_{Y}^{p}\ud s \leq M_C\|x\|_X^p\qquad\text{for all } x \in D(A),\ t \ge 0.
\end{equation}
\end{lemma}

\begin{proof}
If $C$ is $p$-admissible, there exists $t_0 > 0$ and $M > 0$ such that
\begin{equation*}
	\int_0^{t_0}\bigl\|CT(s)x\bigr\|_{Y}^{p}\ud s \leq M\|x\|_X^p\qquad\text{for all } x \in D(A).
\end{equation*}
For $t \leq t_0 $ it is clear that
\begin{align*}
\int_0^{t}\bigl\|CT(s)x\bigr\|_{Y}^{p}\ud s
\leq \int_0^{t_0}\bigl\|CT(s)x\bigr\|_{Y}^{p}\ud s
\leq  M\|x\|_X^p\qquad\text{for all } x \in D(A).
\end{align*}
For $t > t_0$ we can write $t = n t_0 + r$ where $n \in \N$ and $0 \leq r < t_0$. Using \eqref{eq:T(t)-stab} we then obtain
\begin{align*}
\int_0^{t}\bigl\|CT(s)x\bigr\|_{Y}^{p}\ud s
& \le \sum_{k=0}^n \int_{k t_0}^{(k+1)t_0} \bigl\|CT(s)x\bigr\|_{Y}^{p}\ud s \\
& = \sum_{k=0}^n \int_{0}^{t_0} \bigl\|CT(s)T(k t_0)x\bigr\|_{Y}^{p}\ud s \\
& \leq M \sum_{k=0}^n \bigl\|T(k t_0)x\bigr\|_{X}^{p} \\
& \leq M K^p \frac{1}{1- e^{p \omega t_0}} \|x\|_{X}^{p} \qquad \text{for all } x \in D(A).
\end{align*}
Choosing $M_C :=M+ MK^p \frac{1}{1- e^{p \omega t_0}}$ we obtain \eqref{eq:Ctuttit}. This concludes the proof.
\end{proof}

By combining the previous results we obtain the main outcome of this subsection.

\begin{corollary}\label{admC}
If $\A$ is a generator, then $C$ is a $p$-admissible control operator. Conversely, if $C$ is a $p$-admissible control operator, then for every $t\ge0$ the operator $T_{12}(t):D(A)\subset X\to E_1^p$ has a (unique) bounded extension $\Ct:=\overline{T_{12}(t)} \in \sL(X,E_1^p)$. Moreover, $\Ctt$ is strongly continuous and
\begin{equation}\label{termine12}
\Ct x = \1_{[-t,0]}(\cd) C_L T(t+\cd) x \quad \text{for every } x \in X.
\end{equation}
\end{corollary}

\begin{proof}
If $\A$ is the generator of a $C_0$-semigroup $\sTtt$, then by Lemma~\ref{laplC} and the uniqueness of the Laplace transform (see \cite[Prop.~1.7.3]{AB}) we obtain that $\bigl[\sTt\bigr]_{12}x=T_{12}(t)x$ for all $t\ge0$ and $x\in D(A)$. Since $\bigl[\sTt\bigr]_{12}\in\sL(X,E_1^p)$,  Remark~\ref{rem:admiss-T_12} then implies that $C$ is $p$-admissible.

Conversely assume that $C$ is $p$-admissible. Then by Remark~\ref{rem:admiss-T_12} and Lemma~\ref{Ctuttit} each operator $T_{12}(t):D(A)\subset X\to E_1^p$ has a (unique) extension $\Ct\in\sL(X,E_1^p)$. Since by Lemma~\ref{lem:t_12-welldef} the map $t \mapsto \Ct x$ is continuous for every $x \in D(A)$, by a standard density argument (cf. \cite[Lem.~I.5.2]{EN}), $\Ctt$ is strongly continuous.
Finally, using Lemma~\ref{caso:p=1} we obtain \eqref{termine12}.
\end{proof}

\subsection{The Entry \boldmath{$T_{23}(t)$} and Admissible Control Operators}
\label{T23}

We proceed using the same scheme as in the previous subsection and define for $t\ge 0$ the operators
\begin{align*}
&T_{23}(t):E_2^p\to X_{-1},\\
&T_{23}(t)u:=\int_0^t T_{-1}(t-r) B u(r)\ud r \quad \text{for } u \in E_2^p.
\end{align*}

Again we first verify some basic properties of this operator family.

\begin{lemma}\label{lem:t_23-welldef}
For every $u\in E_2^p$ the function $T_{23}(\cd)u:\R_+\to X_{-1}$ is continuous and bounded.
\end{lemma}

\begin{proof}
This follows from \cite[Prop.~1.3.5.(b)]{AB} on the continuity and boundedness of convolutions.
\end{proof}

By the previous result we can consider the Laplace transform of $T_{23}(\cd)u$ in $X_{-1}$.

\begin{lemma}\label{laplB}
For every $\lambda \in \C$ with $\Re \lambda > 0$ and every $u\in E_2^p$ we have
\[
\LT\bigl(T_{23}(\cd)u\bigr) (\lambda)=\RlAme B\hat u(\lambda)=[\RlsA]_{23}u.
\]
\end{lemma}

\begin{proof}
For $u \in E_2^p$ we obtain by Fubini's theorem (see \cite[Thm.~1.1.9]{AB})
\begin{align*}
\LT\bigl(T_{23}(\cd)u\bigr) (\lambda) &= \int_0^{\infty}\int_0^t  e^{-\lambda t} T_{-1}(t-r) B u(r) \ud r \ud t \\
& = \int_0^{\infty}  \int_r^{\infty} e^{-\lambda t} T_{-1}(t-r) B u(r) \ud t \ud r \\
& = \int_0^{\infty}  \int_0^{\infty} e^{-\lambda (t+r)} T_{-1}(t) B u(r) \ud t \ud r \\
& = \RlAme B \hat u(\lambda).\qedhere
\end{align*}
\end{proof}

Next we recall the following well-known notion from control theory (see, e.g., \cite{W-Ad-contr}) which is closely related to the entry $T_{23}(t)$.

\begin{definition}
The control operator $B\in\sL(U,X_{-1})$ is called \emph{$p$-admissible} (with respect to $A$) if there exists $t_0>0$  such that $\rg\bigl(T_{23}(t_0)\bigr)\subset X$.
\end{definition}

\begin{remark}\label{rem-admiss-B}
Note that in any case $T_{23}(t_0) \in \sL (E_2^p, X_{-1})$. Thus if $B$ is $p$-admissible the closed graph theorem implies that $T_{23}(t_0) \in \sL (E_2^p, X)$. On the other hand, if $u\in\WpU$ then using integration by parts we obtain
\[
\int_0^{t_0}T_{-1}(t_0-r) B u(r)\ud r
=A^{-1}_{-1}\biggl(T(t_0)Bu(0)-Bu(t_0) + \int_0^{t_0}T_{-1}(t_0-r) B u' (r)\ud r\biggr)\in X.
\]
Since $\WpU$ is dense in $\LpRpU$ this shows that the operator $B$ is $p$-admissible if and only if there exists $t_0 > 0$ and a constant $M \ge 0$ such that
\[
\biggl\|\int_0^{t_0}T_{-1}(t_0-r)Bu(r)\ud r\biggr\|_X \leq  M\|u\|_{\LpRpU}\qquad\text{for all } u \in \WpU.
\]
Moreover we note that for an operator $B\in\sL(U,X_{-1})$ the condition to be a $p$-admissible control operator gets weaker with growing $p\ge1$.
\end{remark}

Analogously to Lemma~\ref{Ctuttit} we have the following result which is closely related to \cite[Prop.~2.5]{W-Ad-contr}. Here we need again Assumption~\ref{assum-reg}.(i).

\begin{lemma}\label{glmadmB}
If the control operator $B$ is $p$-admissible, then there exists  $M_B\ge0$ such that
\begin{equation}\label{eq:glmadmB}
\biggl\|\int_0^t T_{-1}(t-r) B u(r) \ud r \biggr\|_X \leq M_B \|u\|_{\LpRpU}\qquad\text{for all } u \in \LpRpU,\ t \ge 0.
\end{equation}
\end{lemma}

\begin{proof}
By assumption there exists $t_0 >0 $ and $M > 0$ such that
\[
\biggl\|\int_0^{t_0} T_{-1}(t_0-r) B u(r) \ud r \biggr\|_X \leq M \|u\|_{\LpRpU}\qquad\text{for all } u \in \LpRpU.
\]
For $0 \leq t \leq t_0$ we denote by $u_{t_0 -t}$ the translated function
\begin{equation}\label{eq:u_t}
u_{t_0 -t}(s) :=
\begin{cases}
0 &\text{if } 0 \leq s < t_0-t, \\
u( s -t_0 +t) &\text{if } s \ge t_0-t.
\end{cases}
\end{equation}
Then $u \in \LpRpU$ and $\| u \|_{\LpRpU} = \| u_{t_0-t} \|_{\LpRpU}$.
Moreover
\begin{equation*}
\int_0^t T_{-1}(t-r) B u(r) \ud r = \int_0^{t_0} T_{-1}(t_0 -r) B u_{t_0-t}(r) \ud r \in X.
\end{equation*}
This implies
\begin{align}\nonumber
\biggl\| \int_0^t T_{-1}(t-r) B u(r) \ud r \biggr\|_X & = \biggl\| \int_0^{t_0} T_{-1}(t_0-r) B u_{t_0-t}(r) \ud r \biggr\|_X \\ \label{minori}
 & \leq M \| u \|_{\LpRpU}\qquad\text{for all } u \in \LpRpU.
\end{align}
For $t \ge t_0$ we write $t = nt_0 + s$ for $n \in \N$ and $s \in [0,t_0)$. Then we obtain
\begin{align*}
\int_0^t T_{-1} (t-r)B u(r) \ud r
  &=  \int_0^{s} T_{-1}(nt_0 +s-r) B u(r) \ud r + \int_s^{nt_0 +s}\!\!\!\! T_{-1}(nt_0 +s-r) B u(r) \ud r \\
  &=:L_1+L_2.
\end{align*}
We consider the two terms of the sum separately. For the first one we get $L_1\in X$ and
\begin{align} \label{quellosopra}
\|L_1\|_X 
\le \bigl\| T(nt_0)\bigr\|\cdot\biggl\| \int_0^{s} T_{-1}(s-r) B u(r) \ud r  \biggr\|_X
 & \leq KM \| u\|_{\LpRpU}.
\end{align}
Here  we used again that $\Tt$ is bounded and \eqref{minori}.
For the second term we obtain
\begin{align*}
L_2
 &=  \sum_{k=0}^{n-1}\int_{k t_0}^{(k+1)t_0} T_{-1}(nt_0 -r) B u(r+s) \ud r   \\
 & = \sum_{k=0}^{n-1}T\bigl((n-(k+1))t_0\bigr)\cdot\int_{0}^{t_0} T_{-1}(t_0 -r) B u(r+s+kt_0) \ud r \in X.
\end{align*}
Moreover, using \eqref{eq:T(t)-stab} and  that $B$ is a $p$-admissible control operator this gives the estimates
\begin{align} 
\|L_2\|_X
 \leq  K \sum_{k=0}^{n-1} e^{\omega (n-k-1)t_0}\cdot M \| u \|_{\LpRpU} \label{finaledue}
 \leq \frac{K M}{1- e^{\omega t_0}}\| u \|_{\LpRpU}.
\end{align}
Summing up \eqref{quellosopra} and \eqref{finaledue} we obtain \eqref{eq:glmadmB} for $M_B:= MK + \frac{MK}{1- e^{\omega t_0}}$.
\end{proof}

By combing the previous results we obtain the main statement of this subsection which corresponds to Corollary~\ref{admC}.

\begin{corollary}\label{cor-addm-B}
If $\A$ is a generator, then $B$ is a $p$-admissible control operator. Conversely, if $B$ is a $p$-admissible control operator then for every $t\ge0$ we have $\rg\bigl(T_{23}(t)\bigr)\subset X$ and  $\Bt:=T_{23}(t)\in\sL(E_2^p,X)$. Moreover, the family $\Btt$ is strongly continuous and uniformly bounded.
\end{corollary}

\begin{proof}
If $\A$ is the generator of a $C_0$-semigroup $\sTtt$, then by Lemma~\ref{laplB} and the uniqueness of the Laplace transform 
 we obtain that $T_{23}(t)u=\bigl[\sTt\bigr]_{23}u\in X$ for all $t\ge0$ and $u\in E_2^p$. This implies $\rg(T_{23})(t)\subset X$, thus $B$ is $p$-admissible and $T_{23}\in\sL(E_2^p,X)$.

Conversely, if $B$ is a $p$-admissible control operator, then using Remark~\ref{rem-admiss-B} and Lemma~\ref{glmadmB} we conclude that $\rg\bigl(T_{23}(t)\bigr)\subset X$, hence by the closed graph theorem $\Bt\in\sL(E_2^p,X)$ for every $t \ge 0$.
To show that $\Btt$ is strongly continuous
let $0 \leq r \leq t$ and $u \in \LpRpU$.  Then
\begin{align*}
\left\|\Bt u - \B(r) u \right\|_X & = \left\|\Bt\left(u - u_{t-r}\right) \right\|_X \\
	& \leq \left\|\Bt \right\|\cdot \left\|u-u_{t-r}\right\|_{\LpRpU} \\
	& \leq M_B \left\|u-u_{t-r}\right\|_{\LpRpU},
\end{align*}
where $u_{t-r}$ is defined as in \eqref{eq:u_t}. Since the shift on $\LpRpU$ is strongly continuous, we have
\[
\lim_{|t-r|\rightarrow 0}  \left\|u-u_{t-r}\right\|_{\LpRpU} = 0
\]
and the assertion follows.
\end{proof}

\subsection{The Entry \boldmath{$T_{13}(t)$} and Admissible Pairs of Operators}
\label{T13}

We proceed as in the previous two subsections and start by defining for $t\ge0$ the operators\footnote{We use the notation $T_{13}(t)=T^D_{13}(t)$ to indicate the dependence of this entry on $D\in\sL(U,Y)$.}
\begin{align*}
&T^D_{13}(t):\WzC\subset E_2^p\to E_1^p,\\
\bigl(&T^D_{13}(t)u\bigr)(s):=\1 _{[-t,0]}(s)\biggl( C_L \int_0^{t+s}\!\!\!\!\!T_{-1}(t+s-r) B u(r)\ud r
+ D u(t+s)\biggr)
,\ s\in\R_-,
\end{align*}
where
\[\WzC := \Bigl\{u \in \WzpU: u(0)=u'(0) = 0\Bigr\}.\]

As before we first verify some basic properties of this operator family.

\begin{lemma}\label{lem:t_13-welldef}
For every $u\in \WzC$ the function $T_{13}^D(\cd)u:\R_+\to E_1^p$ is well-defined, continuous and bounded.
\end{lemma}

\begin{proof} 
We first consider the term involving $D\in\sL(U,Y)$. For this it suffices to look at the function $g:\R_+\to E_1^p$ defined by $g(t):=\1 _{[-t,0]}(\cd)v(t+\cd)$ where $v:=Du\in\WzCY$.  Since
\[
\|g(t)\|_{E_1^p}=\int_{-t}^{0}\bigl\|v(t+s)\bigr\|^p_Y\ud s\le\|v\|_{\LpRpY}\quad\text{for all }t\ge0
\]
$g$ is well-defined and bounded. To show its continuity take $0\le r\le t$. Then
\begin{align*}
\bigl\|g(t)-g(r)\bigr\|_{E_1^p}
&=\bigl\|\1_{[-t,0]} (\cd)v(t+\cd)-\1_{[-r,0]} (\cd)v(r+\cd)\bigr\|_{\LpRpY}\\
&\le\biggl(\int_{-t}^{-r}\bigl\|v(t+s)\bigr\|_Y^p\ud s\biggr)^{\frac1p}
+\biggl(\int_{-r}^{0}\bigl\|v(t+s)-v(r+s)\bigr\|_Y^p\ud s\biggr)^{\frac1p}
\\
&= \biggl(\int_{0}^{t-r}\|v(s)\|_Y^p\ud s\biggr)^{\frac1p}+\biggl(\int_{0}^{r}\bigl\|v(t-r+s)-v(s)\bigr\|_Y^p\ud s\biggr)^{\frac1p}
\\&
\to0\quad\text{as $t-r\to 0$},
\end{align*}
where the convergence of the first term follows from the dominated convergence theorem while the second term converges due to the strong continuity of the left-shift semigroup on $\LpRpY$.

Next we consider $T_{13}^0(\cd)u$.
Note that for $u\in\WzC$ using twice integration by parts and the compatibility condition \eqref{eq:Delta-bdd} it follows that for $t+s \ge 0$
\begin{align}\nonumber
\int_0^{t+s}\!\!\!\!\!T_{-1}(t+s&-r) B u(r)\ud r\\\notag
&=-A^{-1}_{-1}\biggl(Bu(t+s) - \int_0^{t+s}\!\!\!\!\! T_{-1}(t+s-r) B u' (r)\ud r\biggr)\\
&=-A^{-1}_{-1}Bu(t+s) -A_{-1}^{-2}Bu'(t+s)+A^{-1}\int_0^{t+s}\!\!\!\!\! T(t+s-r) A_{-1}^{-1} B u''(r)\ud r \label{intxp}
\in D(C_L).
\end{align}
Hence for all $s\in\R_-$ and $t\in\R_+$ the term $\bigl(T^0_{13}(t)u\bigr)(s)$ is well-defined. Moreover, by the same argument as before for the function $g$ it follows that the functions
\begin{align*}
&g_1:\R_+\to E_1^p,&&\kern-60pt g_1(t):=\1 _{[-t,0]}(\cd)C_LA^{-1}_{-1}Bu(t+\cd),\\
&g_2:\R_+\to E_1^p,&&\kern-60pt g_2(t):=\1 _{[-t,0]}(\cd)C_LA^{-2}_{-1}Bu'(t+\cd)
\end{align*}
are bounded and continuous. Since $C_LA^{-1}$ is bounded, to finish the proof it suffices to prove that for $v:=A_{-1}^{-1} B u''\in\LpRpX$ the function
\[
g_3:\R_+\to\LpRmX,\qquad g_3(t):= \1 _{[-t,0]}(\cd)\int_0^{t+\cd}\!\!\!\!\! T(t+\cd-r) v(r)\ud r
\]
is well-defined, continuous and bounded. Applying Young's inequality (see \cite[Prop.1.3.5.(a)]{AB}) to the convolution $T*v$ we get
\begin{align}
\|g_3(t)\|_{\LpRmX}\notag
&=\Bigl(\int_{-t}^0\Bigl\|\int_0^{t+s}\!\!\!\!\! T(t+s-r) v(r)\ud r\Bigr\|^p_X\ud s \Bigr)^{\frac{1}{p}}\\\notag
&= \Bigl( \int_{0}^t\Bigl\|\int_0^{s} T(s-r) v(r)\ud r\Bigr\|^p_X\ud s \Bigr)^{\frac{1}{p}}\\\notag
&\le\bigl\|T*v\bigr\|_{\LpRpX}\\\label{eq:est-conv}
&\le \bigl\| \|T(\cd)\|_{\sL(X)}\bigr\|_{\LeRp}\cdot \|v\|_{\LpRpX}<+\infty,
\end{align}
where in the last step we used \eqref{eq:T(t)-stab}.
This proves that $g_3$ is well-defined and bounded. To show its continuity take  $0\le t_0\le t_1$. Then for $h:=t_1-t_0$ we obtain
\begin{align*}
\|g_3(t_1)-g_3(t_0)\|_{\LpRmX}^p
&=\int_{-t_1}^{-t_0}\Bigl\|\int_0^{t_1+s}\!\!\!\!\! T(t_1+s-r) v(r)\ud r\Bigr\|^p_X\ud s\\
&\qquad  +\int_{-t_0}^{0}\Bigl\|\int_0^{t_1+s}\!\!\!\!\! T(t_1+s-r) v(r)\ud r
                                           -\int_0^{t_0+s}\!\!\!\!\! T(t_0+s-r) v(r)\ud r
\Bigr\|^p_X\ud s\\
&=\int_{0}^{h}\Bigl\|\int_0^{s} T(s-r) v(r)\ud r\Bigr\|^p_X\ud s\\
&\qquad+\int^{t_0}_{0}\Bigl\|\int_0^{s+h}\!\!\! T(s+h-r) v(r)\ud r
                                           -\int_0^{s}T(s-r) v(r)\ud r
\Bigr\|^p_X\ud s\\
&=:L_1^p+L_2^p.
\end{align*}
Moreover, by \cite[Prop.~1.3.4]{AB} the convolution $T*v:\R_+\to X$ is continuous, hence it is uniformly continuous on the compact interval $[0,t_0]$. Thus
\begin{align}\label{eq:est-conv-L2}
L_2^p=\int_0^{t_0}\bigl\|(T*v)(s+h)-(T*v)(s)\bigr\|^p_X\to0\qquad\text{as }h=t_1-t_0\to0.
\end{align}
Furthermore, using the dominated convergence theorem
\begin{align}\label{eq:est-conv-L1}
L_1^p&\to0\qquad\text{as }h=t_1-t_0\to0.
\end{align}
Summing up \eqref{eq:est-conv-L2} and \eqref{eq:est-conv-L1} we complete the proof.
\end{proof}

By the previous result we can Laplace transform $T_{23}(\cd)u$ for $u\in\WzC$. To do so we need the following simple result.

\begin{lemma}\label{lem:L-trafo-conv}
Let $v\in\LpRpX$. Then the convolution $f:=T*v$ is bounded and continuous. Hence for $\Re\lambda>0$ its Laplace transform exists and is given by
\begin{equation*}
\hat f(\lambda)=\RlA\hat v(\lambda).
\end{equation*}
If, in addition, $v\in\LeRpX$ then the same holds for $\Re\lambda\ge0$.
\end{lemma}

\begin{proof} Boundedness of $f$ follows as in \eqref{eq:est-conv} while its continuity is shown in \cite[Prop.~1.3.4]{AB}. Now take $\Re\lambda>0$. Using Assumption~\ref{assum-reg}.(i) the integral
\begin{align}\label{eq:fubini-Re0}
\int_0^{+\infty}\!\int_0^{+\infty}\!e^{-\Re\lambda(t+r)} \bigl\|T(t)v(r)\bigr \|\ud t\ud r\le
K\int_0^{+\infty}\!\int_0^{+\infty}\!e^{-\Re\lambda(t+r)}e^{\omega t} \bigl\|v(r)\bigr \|\ud t\ud r<+\infty
\end{align}
is finite. Hence we can use Fubini's theorem (see \cite[Thm.~1.1.9]{AB}) to conclude that
\begin{align*}
\hat f(\lambda)
&=\int_0^{+\infty}e^{-\lambda t}\int_0^t T(t-r)v(r)\ud r\ud t\\
&=\int_0^{+\infty}\int_r^{+\infty} e^{-\lambda t} T(t-r)v(r)\ud t\ud r\\
&=\int_0^{+\infty}\int_0^{+\infty} e^{-\lambda (t+r)} T(t)v(r)\ud t\ud r\\
&=\RlA\hat v(\lambda).
\end{align*}
Now assume that $v\in\LeRpX$. Then by Young's inequality \cite[Prop.1.3.5.(a)]{AB} we obtain $f\in\LeRpX$.
Hence \eqref{eq:fubini-Re0} still holds for $\Re\lambda=0$ and the claim follows as before.
\end{proof}

\begin{lemma}\label{laplBC}
For every $u\in \WzC$ and $\lambda \in \C$ with $\Re \lambda > 0$ we have
\begin{equation}\label{eq:T_13}
\LT\bigl(T^D_{13}(\cd)u\bigr)(\lambda)=\el\otimes C_L\RlAme B\hat u(\lambda)+\el\otimes D\hat u(\lambda)=[\RlsA]_{13}u.
\end{equation}
\end{lemma}

\begin{proof}
Let $u\in \WzC$ and $\Re \lambda > 0$. Then for $s \in (-\infty, 0]$ we get  using \cite[Thm.~4.2]{Nei:81}
\begin{align*}
\Bigl(\LT\bigl(T^D_{13}(\cd)u\bigr)(\lambda)\Bigr)(s) &= \int_0^{\infty} e^{-\lambda t} \1_{[-t,0]}(s) C_L \int_0^{t+s} T_{-1}(t+s-r) B u(r) \ud r\ud t\\
 &\qquad+ \int_0^{\infty} e^{-\lambda t} \1_{[-t,0]}(s) Du(t+s) \ud t
 =:L_1+L_2.
\end{align*}
We compute the two terms of the sum separately.
For $L_2$ we obtain
\begin{align*}
L_2 &= \int_{-s}^{\infty} e^{-\lambda t} D u(t+s) \ud t
 = e^{\lambda s} D \int_0^{\infty} e^{\lambda t} u(t) \ud r
 =\bigl(\el\otimes D\hat u(\lambda)\bigr)(s).
\end{align*}
Using \eqref{intxp} (for $s=0$), Lemma~\ref{lem:L-trafo-conv} and
\cite[Cor.~1.6.6]{AB}, which states that $\widehat{v'}(\lambda)=\lambda\hat v(\lambda)-v(0)$ for $v\in\WpX$, for the first term we obtain
\begin{align*}
L_1
	&= \int_{-s}^{\infty} e^{-\lambda t} C_L \int_0^{t+s} T_{-1}(t+s-r) Bu(r) \ud r\ud t \\
	&= \int_0^{\infty} e^{-\lambda (t-s)} C_L \int_0^{t} T_{-1}(t-r) Bu(r) \ud r\ud t\\
	&= e^{\lambda s}\int_0^{\infty} e^{-\lambda t}\biggl(-C_L A_{-1}^{-1}Bu(t)
	- C_L A_{-1}^{-2} B u'(t)
     + C_LA^{-1}\int_0^t T(t-r) A_{-1}^{-1}B u''(r) \ud r\biggr)\ud t \\
	&=e^{\lambda s}C_LA^{-1}\Bigl(-Id-\lambda A_{-1}^{-1}+\lambda^2\RlA A^{-1}_{-1}\Bigr)B\hat u(\lambda)\\
&=\el(s) C_L\RlAme B\hat u(\lambda).
\end{align*}
Summing up this gives \eqref{eq:T_13} and completes the proof.
\end{proof}

We proceed by introducing the following  notion closely related to the entry $T^0_{13}(t)$.

\begin{definition}
The pair $(B,C)\in\sL(U,X_{-1})\times\sL(X_1,Y)$ is called \emph{$p$-admissible} (with respect to $A$) if
there exists $t_0>0$ and a constant $M\ge 0$ such that
\begin{equation}\label{est-adm-BC}
\int_0^{t_0}\biggl\|C_L\int_0^{s}T_{-1}(s-r)Bu(r)\ud r\biggr\|_Y^p \ud s\leq  M\|u\|_{\LpRpU}^p
\quad\text{for all $u\in \WzC$.}
\end{equation}
\end{definition}

\begin{remark}\label{rem-admiss-(BC)}
Note that, since for $t\ge0$ and $u\in\WzC$ we have
\begin{align*}
\|T_{13}^0(t)u\|_{E_1^p}^p
=\int_{-t}^{0}\Bigl\|C_L \int_0^{t+s}\!\!\!\!\!T_{-1}(t+s-r) B u(r)\ud r\Bigr\|^p_Y\ud s
=\int_{0}^{t}\Bigl\|C_L \int_0^{s}T_{-1}(s-r) B u(r)\ud r\Bigr\|^p_Y\ud s
\end{align*}
the pair $(B,C)$ is $p$-admissible if and only if $T_{13}^0(t):\WzC\subset X\to E_1^p$ has a (since $\WzC$ is dense in $E_2^p$ necessarily unique) extension in $\sL(E_2^p,E_1^p)$ for some $t>0$. Since $D\in\sL(U,Y)$ this again is equivalent to the fact that $T_{13}^D(t):\WzC\subset X\to E_1^p$ has a bounded extension for some $t>0$.
\end{remark}

Recall that we assume the semigroup $\Tt$ to be exponentially stable. This implies the following result which is analogous to Lemmas~\ref{Ctuttit} and \ref{glmadmB}, and closely related to \cite[Thm.2.5.4.(ii)]{Sta:05} \cite[Prop.~2.1]{reghilbert}.

\begin{lemma}\label{tuttiBC}
Let $B \in \sL(U,X_{-1})$ be a $p$-admissible control operator and $C \in \sL(X_1, Y)$ be a $p$-admissible observation operator. Then the pair $(B,C)$ is $p$-admissible if and only if there exists $M_{BC}\ge0$ such that
\begin{equation}\label{eq:tuttiBC}
\int_0^{t}\biggl\|C_L\int_0^{s}T_{-1}(s-r)Bu(r)\ud r\biggr\|_Y^p \ud s\leq  M_{BC}\|u\|^p_{\LpRpU}
\end{equation}
for all  $u \in \WzC$, $t \ge 0$.
\end{lemma}

\begin{proof}
If the pair $(B,C)$ is $p$-admissible, then we can suppose without loss of generality that in \eqref{est-adm-BC} we have  $t_0=1$.
Then it is clear that \eqref{est-adm-BC} also holds for $t_0$ replaced by some $0 \leq t \leq 1$.
In particular it follows that for every $0 \leq t \leq 1 $ the operator $T_{13}^{0}(t)$ has a (unique) extension $\Ft\in\sL\bigl(\LpRpU,\LpRmY\bigr)$.

Now to prove \eqref{eq:tuttiBC} it suffices to show that it holds for every $t=n \in \N$. 
To this end we write
\begin{align}\notag
\left(\int_0^{n} \biggl\|C_L\int_0^{s}T_{-1}(s-r)Bu(r)\ud r\biggr\|_Y^p \ud s \right)^{\frac{1}{p}} &=  \left(\sum_{k=0}^{n-1}\int_k^{k+1}\biggl\|C_L\int_0^{s}T_{-1}(s-r)Bu(r)\ud r\biggr\|_Y^p \ud s\right)^{\frac{1}{p}} \\\label{eq:est-F(n)}
&\leq \sum_{k=0}^{n-1}\left(\int_0^{1}\biggl\|C_L\int_0^{s+k}T_{-1}(s+k-r)Bu(r)\ud r\biggr\|_Y^p \ud s\right)^{\frac{1}{p}}.
\end{align}
In order to proceed we first estimate the terms of the last sum.
\begin{align*}
\left(\int_0^{1}\biggl\|\right. & \left.C_L \int_0^{s+k}T_{-1}(s+k-r)Bu(r)\ud r\biggr\|_Y^p \ud s\right)^{\frac{1}{p}} \\
& = \left( \int_0^1 \biggl\|C_L \biggl( \sum_{m=0}^{k-1}\int_m^{m+1}T_{-1}(s+k-r)Bu(r)\ud r + \int_k^{s+k} T_{-1}(s+k-r)Bu(r)\ud r \biggr)\biggr\|_Y^p \ud s \right)^{\frac{1}{p}} \\
& \leq  \left(\int_0^1 \biggl\| C_L  T(s) \sum_{m=0}^{k-1}  T(k-m-1) \int_m^{m+1}T_{-1}(m+1-r)Bu(r)\ud r \biggr\|_Y^p \ud s\right)^{\frac{1}{p}}\\
&\qquad + \left(\int_0^1 \biggl\| C_L\int_0^{s} T_{-1}(s-r)Bu(r+k)\ud r \biggr\|_Y^p \ud s\right)^{\frac{1}{p}}=:L_1+L_2.
\end{align*}
We consider the two terms of this sum separately. To this end we define for $m \in \N$ the operators $P_m\in\sL\bigl(\LpRpU\bigr)$ by $(P_m u)(s) := \1_{[0,1]}(s) u(s+m)$ for $s \in [0,\infty)$. Then
\begin{equation*}
L_2
=\bigl\| \bF(1) P_k u \bigr\|_{\LpRmY}%
\leq  M \| P_k u \|_{\LpRpU},
\end{equation*}
where we used that the pair $(B,C)$ is $p$-admissible.
The first term of the sum can be estimated as
\begin{align*}
L_1
& \leq  M_C^{\frac{1}{p}}\, \biggl\| \sum_{m=0}^{k-1} T(k-m-1) \int_0^{1}T_{-1}(1-r)Bu(r+m)\ud r \biggr\|_X \\
& \leq   M_C^{{\frac{1}{p}}} K\sum_{m=0}^{k-1} e^{\omega(k-m-1)}\bigl\| \B(1)P_m u \bigr\|_X \\
& \leq M_C^{{\frac{1}{p}}}M_B K \sum_{m=0}^{k-1}  e^{\omega(k-m-1)} \|P_m u\|_{\LpRpU}.
\end{align*}
Here we used that $C$ is a $p$-admissible observation operator, the stability condition~\eqref{eq:T(t)-stab} and that $B$ is a $p$-admissible control operator.
Thus introducing the  notation
\[
l_m :=
\begin{cases}
M & \text{if }m=0 \\
M_C^{{\frac{1}{p}}}M_B K  e^{\omega(m-1)} & \text{if } 1 \leq m \leq n-1
\end{cases}
\]
we obtain that for $0\leq k \leq n-1$
\[
\left(\int_0^{1}\biggl\|C_L \int_0^{s+k}T_{-1}(s+k-r)Bu(r)\ud r\biggr\|_Y^p \ud s\right)^{\frac{1}{p}} \leq \sum_{m=0}^k l_{k-m} \|P_m u\|_{\LpRpU}.
\]
Summing up we obtain by \eqref{eq:est-F(n)} for arbitrary $n\in\N$ and $u\in\WzC$ that
\begin{align*}
\left(\int_0^{n} \biggl\|C_L\int_0^{s}T_{-1}(s-r)Bu(r)\ud r\biggr\|_Y^p \ud s \right)^{\frac{1}{p}}
&\le\sum_{k=0}^{n-1}\sum_{m=0}^k l_{k-m} \|P_m u\|_{\LpRpU}\\
&\leq \Bigl(\sum_{k=0}^{n-1} l_k\Bigr)\cdot\Bigl(\sum_{k=0}^{n-1}\|P_ku\|_{\LpRpU}^p\Bigr)^{\frac1p}\\
&\le\Bigl(M +  \frac{M_C^{\frac{1}{p}}M_B  K}{1-e^{\omega}}\Bigr)\cdot
\|u\|_{\LpRpU}\\
&=:M_{BC}\cdot\|u\|_{\LpRpU},
\end{align*}
where in the second estimate we used Young's inequality for the convolution of sequences.
\end{proof}

\begin{remark}\label{F}
If $B$ and $C$ are both $p$-admissible then by Lemma~\ref{tuttiBC} the pair $(B,C)$ is $p$-admissible if and only if the operator
\begin{align*}
&\bFt:  \WzC \subset \LpRpU \to \LpRpY,\\
(&\bFt u)(\cd) := C_L \int_0^\cd T_{-1}(\cd-r) B u(r) \ud r
\end{align*}
has a bounded extension in $\sL\bigl(\LpRpU,\LpRpY\bigr)$.
\end{remark}

Combining the previous results we obtain the main statement of this subsection.

\begin{corollary}\label{cor:BC-admiss}
If $\A$ is a generator, then $(B,C)$ is $p$-admissible. Conversely, if $B$, $C$ and the pair $(B,C)$ are all $p$-admissible, then for every $t\ge0$ the operator $T^D_{13}(t):\WzC\subset E_2^p\to E_1^p$ has a (unique) bounded extension $\Ft:=\overline{T^D_{13}(t)} \in \sL(E_2^p,E_1^p)$ and $\Ftt$ is strongly continuous.
\end{corollary}

\begin{proof}
If $\A$ is the generator of a $C_0$-semigroup $\sTtt$, then by Lemma~\ref{laplBC} and the uniqueness of the Laplace transform we obtain that $\bigl[\sTt\bigr]_{13}u=T_{13}(t)u$ for all $t\ge0$ and $u\in\WzC$. Since $\bigl[\sTt\bigr]_{13}\in\sL(E_2^p,E_1^p)$,  Remark~\ref{rem-admiss-(BC)} then implies that the pair $(B,C)$ is $p$-admissible.

Conversely,
if $(B,C)$ is $p$-admissible, then Remark~\ref{rem-admiss-(BC)} and Lemma~\ref{tuttiBC} imply that the operator $T^D_{13}(t):\WzC\subset X\to E_1^p$ has a (unique) bounded extension for every $t\ge0$.
We recall that by Lemma~\ref{lem:t_13-welldef} the map $t \mapsto T_{13}^D(t)u$ is continuous for every $u \in \WzC$.
Hence by a density argument, using Lemma~\ref{tuttiBC}, we conclude that $\Ftt$ is strongly continuous.
\end{proof}

\section{Characterization of Admissible Pairs in the Frequency Domain}
\label{sec:FM}

The aim of this section is to characterize admissibility in the frequency domain, i.e., in terms of the entries of the resolvent $\RlsA$ of $\A$. For the admissibility of the observation operator $C$ (related to the boundedness of  the entry $T_{12}(t)$ of the semigroup operators $\sT(t)=(T_{jk}(t))_{3\times 3}$, cf. Subsection~\ref{T12}) and the admissibility of the control operator $B$ (related to the boundedness of  $T_{23}(t)$, cf. Subsection~\ref{T23}) this problem was posed by Weiss in \cite{Wei:91}, \cite{Wei:98} and in the sequel has been studied by various authors. We refer to \cite{JP:04} for a nice survey on this matter.

Here we concentrate on the entry $T_{13}(t)$ related to the admissibility of the pair $(B,C)$.
Our approach is based on the concept of Fourier multipliers, cf.  \cite[Sect.~2.5]{Amann}, \cite[Sect.~5.2]{BP}, \cite[App.~E.1]{Haase}. First we recall the basic definition,  where $\Four$ denotes the Fourier transform.

\begin{definition}
Let $V,W$ be two Banach spaces and $1 \leq p < \infty$.
A function $ m \in\L^{\infty}\bigl(\R,\sL(V,W)\bigr)$ is called (bounded) $\L^p$-\emph{Fourier multiplier} if the map\footnote{By $\Schw(\R,V)$ we denote the set of all Schwartz functions with values in the Banach space $V$.
}
\[
v \mapsto \Four^{-1} \left( m \Four v\right)
\qquad  \text{for } v \in \Schw\left(\R, V\right)
\]
has a continuous extension to a bounded operator from $\L^p(\R,V)$ to $\L^p(\R,W)$.
\end{definition}

Since by Assumption~\ref{assum-reg}.(i) we have $i\R\subset\rho(A)$ we can, using \eqref{eq:Delta-bdd}, define the map
\[
m_{13}:\R\to\sL(U,Y), \quad m_{13}(\gamma):= C_L R(i\gamma,A_{-1})B.
\]
Now the following characterization holds.

\begin{proposition}\label{m13}
Let $B$ and $C$ be $p$-admissible control and observation operators, respectively. Then
the pair $(B,C)$ is $p$-admissible if and only if $m_{13}$ is a bounded Fourier multiplier.
\end{proposition}

\begin{proof}
As we have seen in Remark~\ref{F},
the pair $(B,C)$ is $p$-admissible if and only if the operator $\bFt$ has a bounded extension to $\LpRpU$.
Let $\gamma\in\R$ and
\[u \in W^{2,p}_{0,c}([0,\infty),U) := \Bigl\{u \in \WpU: u(0)=u'(0)=0\text{ and $u$ has compact support}\Bigr\}\]
For such $u$ we have $u$, $u'$, $u''\in\LeRpU$. Hence even though $\Re(i\gamma)=0$, we can argue as in the proof of Lemma~\ref{laplBC}, using the second part of Lemma~\ref{lem:L-trafo-conv}, to obtain
\[
\sL(\bFt u)(i \gamma) =C_L R(i\gamma, A_{-1}) B \sL(u)(i \gamma)\qquad\text{for all }\gamma\in\R.
\]
It thus follows
\begin{equation*}
\Four (\bFt u) = m_{13} \Four u.
\end{equation*}
Using this we conclude that
$m_{13}$ is a bounded Fourier-multiplier if and only if $\bFt$ has a bounded extension to $\LpRpU$ if and only if the pair $(B,C)$ is $p$-admissible.
\end{proof}


\section{Well-Posed Linear Control Systems and the Lax-Phillips Semigroup}
\label{A-csu}

We now sum up the findings of Subsections~\ref{T12}--\ref{T13} to obtain our main result.
For a linear control system \ref{csu} verifying Assumption~\ref{assum-reg} the following holds.

\begin{theorem}\label{char-gen-sA-admiss}
The system \ref{csu} is $p$-well-posed on $\Xp$, i.e., $\sA$ is the generator of a $C_0$-semigroup $\sT=\sTtt$ on $\Xp$, if and only if $B$ is a $p$-admissible control operator, $C$ is a $p$-admissible observation operator and also the pair $(B,C)$ is $p$-admissible. In this case the semigroup $\sT$ is given by
\begin{equation}\label{laxph}
\sTt =
\begin{pmatrix}
S_1(t) & \Ct &  \Ft \\
0 & T(t) & \Bt\\
0 & 0 & S_2(t)
\end{pmatrix}\qquad\text{for all } t \ge 0.
\end{equation}
\end{theorem}

\begin{proof}
If $\sA$ is a generator on $\Xp$ then $C$, $B$ and $(B,C)$ are $p$-admissible by Corollaries~\ref{admC}, \ref{cor-addm-B} and \ref{cor:BC-admiss}, respectively. Conversely, if $C$, $B$ and $(B,C)$ are $p$-admissible, again by Corollaries~\ref{admC}, \ref{cor-addm-B} and \ref{cor:BC-admiss} we can define a strongly continuous operator family $\sTtt$ by \eqref{laxph}. Using Lemmas~\ref{laplC}, \ref{laplB} and \ref{laplBC} it follows that
\[
\sL\bigl(\sT(\cd)\bigr)(\lambda)=\RlsA
\]
hence by Lemma~\ref{lem-char-ABHN}, $\sT$ is a $C_0$-semigroup with generator $\sA$.
\end{proof}

Combining Proposition~\ref{m13} with Theorem~\ref{char-gen-sA-admiss} we obtain a second characterization.

\begin{corollary}\label{admBC}
The system \ref{csu} is $p$-well-posed on $\Xp$ if and only if $B$, $C$ are $p$-admissible control and observation operators, respectively, and $m_{13}$ is a bounded Fourier multiplier.
\end{corollary}

As a corollary we characterize the $2$-well-posedness of the system \ref{csu} in case all the spaces $X$, $Y$ and $U$ are Hilbert spaces.
Using the Plancharel Theorem (see \cite[Thm.1.8.2]{AB}) one can first prove the following.

\begin{lemma}\label{multhilbert}
Let $V, W$ be two Hilbert spaces, then every $m \in\L^{\infty}\bigl(\R, \sL(V,W)\bigr)$ is a (bounded) $\L^2$-Fourier multiplier.
\end{lemma}

Combining Corollary~\ref{admBC} and Lemma~\ref{multhilbert} we immediately obtain our next result.

\begin{corollary}\label{hilbert}
Let $X$, $Y$ and $U$ be Hilbert spaces. Then the system \ref{csu} is $2$-well-posed if and only if $B$, $C$ are $2$-admissible and $m_{13} = C_L R(i\,\cd, A_{-1}) B \in \L^{\infty}\bigl(\R, \sL(U,Y)\bigr)$.
\end{corollary}

\begin{remark} The semigroup $\sTtt$ in \eqref{laxph} already appears in Staffans and Weiss \cite[Prop.~6.2]{tesi2} and there is called the \emph{Lax-Phillips semigroup} (of index 0) referring to the paper \cite{LP} by Lax and Phillips.

This semigroup describes the solutions of the well-posed system \ref{csu} as follows. For $\x=\bigl(y(\cd),x,u(\cd)\bigr)^t\in \Xp$
\begin{itemize}
	\item the first component of $\sT(\cd)\x$ gives the \emph{past output},
	\item the second component of $\sT(\cd)\x$ represents the \emph{present state},
	\item the third component of $\sT(\cd)\x$ can be interpreted as the \emph{future input}
\end{itemize}
of the system.
\end{remark}

\begin{remark}\label{rem:not-stable}
If the semigroup $\Tt$ generated by the state operator $A$ is not exponentially stable as supposed in Assumption~\ref{assum-reg}.(i)  (i.e., if the growth bound $\omega_0(A)\ge0$, cf.  \cite[Def.~I.5.6]{EN}) we choose $\lambda_0 > \omega_0(A)$. Then for the rescaled generator $A - \lambda_0$ we obtain $\omega_0 (\Al) < 0$.
Moreover, on the product space $\Xp$ we introduce the operator matrix $\sAl$ associated to the control problem $\Sigma(\Al,B,C,D)$. This operator can be written as
\[
\sAl = \A - \lambda_0\sP_2\qquad\text{for}\qquad
\sP_2:=
\left(\begin{smallmatrix}
0 & 0 & 0\\
0 & 1 & 0\\
0 & 0 & 0
\end{smallmatrix}\right)\in\sL(\Xp).
\]
If there exists  $\lambda \in \rho(A)$ such that $\rg\bigl(R(\lambda, A_{-1})B\bigr) \subset D(C_L)$ then this holds for every $\lambda\in\rho(A)$.  Hence $\rg\bigl(R(\mu, A_{-1}-\lambda_0)B\bigr)=\rg\bigl(R(\mu+\lambda_0, A_{-1})B\bigr) \subset D(C_L)$ for every $\mu \in \rho(\Al)$. This shows that $A$ satisfies the compatibility assumption \eqref{bild} if and only if $A-\lambda_0$ does.
\end{remark}

Moreover we have the following result.

\begin{theorem}\label{pert} Let $\lambda_0\in\rho(A)$. Then the following are equivalent.
\begin{enumerate}
	\item[(a)] $\A$ is the generator of a  $C_0$-semigroup on $\Xp$,
	\item[(b)] $\sAl$ is the generator of a  $C_0$-semigroup on $\Xp$,
	\item[(c)] $B,\,C$ and the pair $(B,C)$ are $p$-admissible with respect to $\Al$,
	\item[(d)] $B$ and $C$ are $p$-admissible with respect to $\Al$ (or $A$) and $m_{13}^{\lambda_0}:= C_L R(\lambda_0 +i\,\cd, A_{-1})B $ is a bounded Fourier-multiplier.
\end{enumerate}
\end{theorem}
\begin{proof}
(a)$\iff$(b). Since $\sA$ and $\sAl$ differ only by a bounded operator this equivalence holds by the bounded perturbation theorem, cf. \cite[Thm.III.1.3]{EN}.

(b)$\iff$(c).  This equivalence holds by Theorem \ref{char-gen-sA-admiss}.

(c)$\iff$(d). It is clear that $B$ and  $C$ are $p$-admissible with respect to $\Al$ if and only if they are $p$-admissible with respect to $A$. By Theorem \ref{m13} the pair $(B,C)$ is $p$-admissible with respect to $\Al$ if and only if $m_{13}^{\lambda_0} = C_L R(i\,\cd, A_{-1}-\lambda_0)B=C_L R(\lambda_0 +i\,\cd, A_{-1})B $ is a bounded Fourier-multiplier.
\end{proof}

\section{Example: A Heat Equation with Boundary Control and Point Observation}
\label{example}

To illustrate our results we consider a metal bar of length $\pi$ modeled as a segment $[0,\pi]$.
Our aim is to control its temperature by putting controls $u_0(t)$ and $u_1(t)$ at the edges $0$ and $\pi$.
Moreover, we observe the system by measuring its temperature at the center $\frac{\pi}{2} \in [0,\pi]$.

As state space we choose  the Hilbert space $X = \L^2 [0,\pi]$ and consider the \emph{state function} $x(s,t)$ representing the temperature in the point $s\in [0,\pi]$ at time $t\ge0$.

If we start from the temperature profile $x_0\in X$, the time evolution of our system can be described by a heat equation with boundary control and a point observation, more precisely by
\begin{equation}\label{eq:1}
\begin{cases}
\frac{\partial x(s,t)}{\partial t} = \frac{{\partial}^2 x(s,t)}{\partial s^2},&t \ge 0,\ s\in[0,\pi], \\
x(s,0) = x_0(s),&s \in [0,\pi],\\
\frac{\partial x}{\partial s}(0,t) = u_0 (t),&t \ge 0,\\
\frac{\partial x}{\partial s}(\pi,t) = u_1 (t),&t \ge 0,\\
y(t) = x\bigl(\tfrac{\pi}{2},t\bigr),&t \ge 0.
\end{cases}
\end{equation}
Here the boundary conditions in $s=0$ and $s=\pi$ involving $u_0(\cd)$ and $u_1(\cd)$ describe the heat exchange between the ends of the bar and the environment.

In order to write \eqref{eq:1} as a linear control system of the form \ref{csu} we use the approach for boundary control problems developed in \cite[Sect.~2]{maximal}. To this end we define the following operators and spaces.
\begin{itemize}
\item The \emph{maximal system operator}
\[
A_m := \frac{\ud^2}{\ud s^2}
\qquad\text{with domain}\qquad
D(A_m) := W^{2,2}[0,\pi]\subset X=\L^2 [0,\pi];
\]
\item the \emph{boundary space} $\partial X:=\C^2$ and the \emph{boundary operator}\footnote{Here $\bigl[D(A_m)\bigr]$ indicates the space $D(A_m)$ endowed with the graph norm $\|\cd\|_{A_m}$.}
\[
Q:\bigl[D(A_m)\bigr]\to\partial X,\quad Qf:={\bigl( f'(0), \ f'(\pi)  \bigr)}^t;
\]
\item the \emph{control space} $U:=\C^2$ and the \emph{control operator} $\tilde{B}:= Id\in\sL(U,\partial X)$;
\item the \emph{observation space} $Y:= \C$ and the \emph{observation operator}\footnote{By $\delta_{\frac{\pi}{2}}$ we indicate the point evaluation in $\frac{\pi}{2}$.} $C := \delta_{\frac{\pi}{2}}$.
\end{itemize}

With this notation  \eqref{eq:1} can be rewritten as an abstract Boundary Control System
\begin{equation*}
\tag*{(aBCS)}
\label{eq:bc}
\begin{cases}
\dot{x}(t)=A_m x(t),& t\ge0, \\
Qx(t)=\tilde{B}u(t),& t\ge0,\\
y(t) = C x(t),& t \ge 0,\\
x(0)=x_0.
\end{cases}
\end{equation*}
\goodbreak
We note that
\begin{itemize}
\item the operator $A\subset A_m$ with domain
\begin{equation*}
D(A) := \ker(Q)=\bigl\{h\in\W^{2,2}[0,\pi]:f'(0) = f'(\pi) = 0\bigr\}
\end{equation*}
is the generator of a $C_0$-semigroup $\Tt$ on $X$, and its spectrum is given by $\sigma(A) = \bigl\{-n^2 : n \in \N\bigr\}$ (see \cite[Sect.~II.3.30]{EN});
\item the boundary operator $Q$ is surjective,
\end{itemize}
i.e., the Main Assumptions~2.3 in \cite{maximal} are satisfied. In order to use the abstract theory for boundary control systems developed in \cite[Sect.~2]{maximal} we need the \emph{Dirichlet operator}
\[
Q_{\lambda} = {\bigl(Q_{|\ker(\lambda-A_m)}\bigr)}^{-1}:\partial X\to \ker(\lambda-A_m)
\]
which by \cite[Lem.~2.4.(ii)]{maximal} exists for every $\lambda \in \rho(A)$.

Since\footnote{With $\spn\{f, g\}$ we denote the linear vector space generated by $f$ and $g$.} $\ker(\lambda-A_m) = \spn\bigl\{\cosh\bigl(\sqrt{\lambda}\,\cd\bigr), \cosh\bigl(\sqrt{\lambda}(\pi-\cd)\bigr)\bigr\}$,
a simple computation  shows that
\begin{equation*}
Q_{\lambda} = \bigl(q_0(\cd), q_1(\cd)\bigr),
\end{equation*}
where for $s \in [0,\pi]$
\begin{equation*}
	q_0(s):= -\frac{\cosh\bigl(\sqrt{\lambda}(\pi-s)\bigr)}{\sqrt{\lambda}\sinh\bigl(\sqrt{\lambda}\pi\bigr)},\qquad\quad
	q_1(s):=
\frac{\cosh\bigl( \sqrt{\lambda}s\bigr)}{\sqrt{\lambda}\sinh\bigl( \sqrt{\lambda}\pi\bigr)}.
\end{equation*}
Let $B_{\lambda} := Q_{\lambda}\tilde{B}=Q_{\lambda}$. Then by \cite[Sect.~2]{maximal} the system \ref{eq:bc} is equivalent to \ref{csu} for the operators
\begin{align*}
B & := (\lambda-A_{-1})Q_{\lambda} \in \sL (U, X_{-1}), \\
C & := \delta_{\frac{\pi}{2}} \in \sL(X_1, Y), \\
D & := 0 \in \sL(U,Y).
\end{align*}

In order to prove $2$-well-posedness of the system $\Sigma(A,B,C,0)$ we transform it into an isomorphic problem on $\ell^2$.

To this end we first note that $A$ is self-adjoint and has compact resolvent. Hence its normalized eigenvectors given by
\[
e_n(s) = \sqrt{\tfrac{\den}{\pi}}\, \cos(ns) \quad
\qquad\text{where}\qquad
\den=
\begin{cases}
1&\text{if }n=0,\\
2&\text{if }n\ge1
\end{cases}
\]
form an orthonormal basis of $X$. Using this basis we define the surjective isometry
\[
J:X\to\ell^2,\quad Jf:=\bigl(\left\langle f, e_n\right\rangle\bigr)_{n\in\N},
\]
which associates to a function $f\in X$ the sequence of its Fourier coefficients relatively to $(e_n)_{n\in\N}$.

Next we put $z(t):=Jx(t)$. Then the system $\Sigma(A,B,C,0)$ transforms to
\[
\Sigma\bigl(JAJ^{-1},JB,CJ^{-1},0\bigr)=\Sigma\bigl(JAJ^{-1},J(\lambda-A_{-1})Q_{\lambda},\delta_{\frac\pi2}J^{-1},0\bigr).
\]
In particular, the differential operator $A$ transforms into the multiplication operator
\[
JAJ^{-1}=: M_{\alpha} =: M  : D(M)\subset \ell^2 \to \ell^2
\]
where $\alpha = (-n^2)_{n\in\N}$ and
\[
D(M) = \Bigl\{(a_n)_{n\in\N} \in \ell^2: (-n^2a_n)_{n\in\N} \in \ell^2 \Bigr\}.
\]
This gives for $\lambda>0$ the extrapolation space
\[
X_{-1}^M = \Bigl\{(a_n)_{n\in\N} \in {\C}^{\N}: \Bigl(\frac{a_n}{\lambda+n^2}\Bigr)_{n\in\N} \in \ell^2 \Bigr\}.
\]
Moreover, the Dirichlet operator $Q_{\lambda}$ transforms into the operator
\[
JQ_{\lambda} = \biggl(\Bigl(-\frac{\sqrt{\den/\pi}}{\lambda+n^2}\Bigr)_{n\in\N}, \Bigl(\frac{(-1)^n\sqrt{\den/\pi}}{\lambda+n^2}\Bigr)_{n\in\N}\biggr).
\]
Thus the control operator $B$ transforms into
\begin{equation}\label{eq:def-b}
b := J(\lambda-A_{-1})Q_{\lambda}=(\lambda-M)JQ_{\lambda} =  \biggl(\Bigl(-\sqrt{{\tfrac{\den}{\pi}}}\Bigr)_{n\in\N}, \Bigl((-1)^n\sqrt{{\tfrac{\den}{\pi}}}\Bigr)_{n\in\N}\biggr),
\end{equation}
while the observation operator $C$ transforms into the operator
\begin{equation}\label{eq:def-c}
c  := C J^{-1} = \Bigl(e_n\bigl(\tfrac{\pi}{2}\bigr)\Bigr)_{n \in \N}
\end{equation}
where
\[ e_n\bigl(\tfrac{\pi}{2}\bigr)
=\begin{cases}
\ 0&\text{if $n$ is odd},\\
(-1)^{\frac n2}\sqrt{\tfrac{\den}\pi}&\text{if $n$ is even.}
\end{cases}
\]
Summing up, the Control System \eqref{eq:1} is isometrically isomorphic to
\begin{equation}
\label{eq:extr-iso}
\begin{cases}
\dot{z}(t) = M z(t)+bu(t),& t\ge0, \\
y(t) = c z(t), & t\ge 0,\\
z(0)=z_0,&
\end{cases}
\end{equation}
where $z(t):=Jx(t)\in\ell^2$ and $z_0:=Jx_0$.

Our aim is now to prove the $2$-well-posedness of the system $\Sigma(M,b,c,0)$ in \eqref{eq:extr-iso}.  Since $\omega_0(A)=\omega_0(M)=0$ we consider in the sequel   $M-1$ instead of $M$, cf. Remark~\ref{rem:not-stable} and Theorem~\ref{pert}.

First we verify the compatibility condition \eqref{bild}.

\begin{lemma} \label{lem:comp-expl}
For every $\gamma\in\R$ we have
\begin{equation}\label{bild-e}
\rg \bigl(R(1+i\gamma,M_{-1})b\bigr) \subset D(c_L).
\end{equation}
Moreover, $m_{13}(\cd):= c_L R(1+i\,\cd,M_{-1})b \in \L^{\infty}\bigl(\R,\sL(U,Y)\bigr)=\L^{\infty} \bigl(\R,\sL(\C^2,\C)\bigr)$.
\end{lemma}

\begin{proof}
Let $u :=\binom{u_1}{u_2} \in U=\C^2$ and $\gamma\in\R$. Then it follows
\[
R(1+i\gamma,M_{-1})bu = \biggl(\frac{1}{1+n^2+i\gamma}\Bigl(-\sqrt{{\tfrac{\den}{\pi}}}\,u_1 + (-1)^n\sqrt{{\tfrac{\den}{\pi}}}\, u_2\Bigr)\biggr)_{n\in \N}=:(r_n)_{n\in\N}.
\]
Since
\[
\bigl|e_n\bigl(\tfrac{\pi}{2}\bigr)\cdot r_n\bigr| \leq \frac{4}{(1+n^2)\pi}\cdot\bigl(|u_1|+|u_2|\bigr)\qquad\text{for all } n \in \N,\,\gamma\in\R
\]
the series
\[
\sum_{n=0}^\infty e_n\bigl(\tfrac{\pi}{2}\bigr)\cdot r_n
\]
converges.  By \cite[Prop.~7.2]{lebext} this implies \eqref{bild-e} and
\[
\bigl|c_LR(1+i\gamma,M_{-1})bu \bigr|
\le \frac{4\sqrt{2}}{\pi}\sum_{n=0}^\infty\frac{1}{1+n^2}\cdot\|u\|_2
\qquad\text{for all }\gamma\in\R,\, u\in U.
\]
Since this implies that $m_{13}(\cd)$ is bounded the proof is complete.
\end{proof}

Next we verify the $2$-admissibility of the operators $c$ and $b$. To this end we denote by $\St$ the semigroup generated by $M-1$.

\begin{proposition}
The observation operator $c$  is $2$-admissible with respect to $M-1$.
\end{proposition}

\begin{proof} Let $t_0>0$ and $z=(z_n)_{n\in\N}\in D(M)$. Then by the Cauchy--Schwarz inequality we obtain
\begin{align*}
\int_0^{t_0}\bigl|c\,S(s)z\bigr|^2\ud s
&=\int_0^{t_0}\biggl|\sum_{n=0}^\infty e_n(\tfrac\pi2)\, e^{-(1+n^2)s}z_n\biggr|^2\ud s\\
&\le\frac2\pi\sum_{n=0}^{+\infty}\int_0^{+\infty}e^{-2(1+n^2)s}\ud s\cdot \sum_{n=0}^{+\infty}|z_n|^2\\
&\le\frac1\pi\sum_{n=0}^{+\infty}\frac1{1+n^2}\cdot\|z\|_{\ell^2}^2,
\end{align*}
hence by definition $c$ is an admissible observation operator.
\end{proof}

\begin{proposition}
The control operator $b = (b_1,b_2)$  is $2$-admissible with respect to $M-1$.
\end{proposition}

\begin{proof}
Clearly $b$ is $2$-admissible if and only if $b_1,\,b_2:\C\to X_{-1}$ are both $2$-admissible. Let $t_0>0$ and $u\in\L^2[0,+\infty)$. Then by Young's inequality (cf. \cite[Prop.~1.3.5.(a)]{AB}) we obtain for $i=1,2$
\begin{align*}
\left\|\int_0^{t_0}S_{-1}(t_0-r)b_iu(r)\ud r\right\|_{\ell^2}^2
&\le\frac2\pi\sum_{n=0}^{+\infty}\biggl(\int_0^{t_0}e^{-(1+n^2)(t_0-r)}\bigl|u(r)\bigr|\ud r\biggr)^2\\
&\le\frac2\pi\sum_{n=0}^{+\infty}\biggl(\int_0^{+\infty}e^{-2(1+n^2)r}\ud r\biggr)^2\cdot\biggr(\int_0^{+\infty}\bigl|u(r)\bigr|^2\ud r\biggr)^2\\
&=\frac1\pi\sum_{n=0}^{+\infty}\frac1{1+n^2}\cdot\|u\|_{\L^2[0,+\infty)}^2,
\end{align*}
hence by definition $b_i$ is an admissible control operator.
\end{proof}

\begin{remark} For multiplication semigroups and finite dimensional observation/control spaces there exists a characterization for the admissability of an observation/control operator via a \emph{Carleson measure criteria}.  For the details we refer to \cite[Thm.~5.3.2]{Weiss} and \cite[Cor.~2.5]{rusho}, \cite[Thm.~1.2]{diagon}, respectively.

\end{remark}

Finally, from Lemmas~\ref{m13}, \ref{multhilbert} and \ref{lem:comp-expl} we obtain the following.

\begin{corollary}
The pair $(b,c)$ is $2$-admissible.
\end{corollary}

Summing up we obtain by Theorem~\ref{pert} the main result of this section.

\begin{corollary}
The system $\Sigma(M,b,c,0)$, hence also the Heat Equation~\eqref{eq:1}, is $2$-well-posed.
\end{corollary}

\appendix
\section{}
We used several times the following simple result which relates the existence of a bounded extension of a densely defined operator to a range condition. It allowed us to characterize admissibility by a range condition or, alternatively, by a boundedness condition on a dense set.

\begin{lemma}\label{lem-app}
Let  $V$, $W$, $Z$ be arbitrary Banach spaces, $D\subset V$ be a dense subspace and
assume that $W\inc Z$ is continuously embedded. Then for linear operators
\begin{align*}
\tilde Q:{}&\phantom{D\subset}\;V\to Z, \qquad \tilde Q \in \sL(V,Z), \\
Q:{}&D\subset V\to W,
\end{align*}
the following assertions are equivalent.

\begin{enumerate}
\item[(a)] There exists $M\ge0$ such that
\[
\|Qv\|_W\le M\cdot\|v\|_V\quad\text{for all }v\in D
\]
and $\tilde Q$ is the unique bounded extension of $Q$.
\item[(b)] $Q=\tilde Q_{|D}$ and $\rg(\tilde Q)\subset W$.
\end{enumerate}
In this case, $\tilde Q\in\sL(V,W)$.
\end{lemma}

\begin{proof}
(a) $\Rightarrow$ (b). It is clear that $Q$ has a unique bounded extension $\bar Q \in \sL (V,W)$. Since $\bar Q = \tilde Q$ it follows $\rg (\tilde Q) \subset W$.

(b) $\Rightarrow$ (a). Since $\rg(\tilde Q)\subset W$  the closed graph theorem implies that $\tilde Q \in \sL (V,W)$.	
Hence there exists a constant $M \ge 0$ such that for all $ v \in D$
\[\| Q v \|_W = \| \tilde Q v \|_W \leq M  \| v \|_V.\]
Moreover, $\tilde Q$ is the unique bounded extension of $Q$.
As claimed in this case $\tilde Q \in \sL (V,W)$.
\end{proof}


\emph{Miriam Bombieri\\
Arbeitsbereich Funktionalanalysis \\
Mathematisches Institut\\
Auf der Morgenstelle 10 \\
D-72076 T\"{u}bingen\\
mibo@fa.uni-tuebingen.de
}

\emph{Klaus Engel\\
Università degli Studi dell'Aquila \\
Dipartimento di Ingegneria e Scienze dell'Informazione e Matematica\\
Via Vetoio (Coppito 1)\\
67100 L'Aquila\\
engel@ing.univaq.it
}


\begin{thebibliography}{14}
\normalsize

\bibitem{Amann} H.~Amann.
\emph{Vector-Valued Distributions and Fourier Multipliers}.
Unpublished Manuscript, 2003. Available at: \url{www.math.uzh.ch/amann/files/distributions.ps}.

\bibitem{AB}
W.~Arendt, C.J.K.~Batty, M.~Hieber, F.~Neubrander.
\emph{Vector-Valued Laplace Transforms and Cauchy Problems}.
Birkh\"auser, 2001.

\bibitem{BP}
A.~B\'atkai, S.~Piazzera.
\emph{Semigroups for Delay Equations}.
Research Notes in Mathematics, 2005.

\bibitem{CW:89}
R.F.~Curtain, G.~Weiss.
\emph{Well posedness of triples of operators (in the sense of linear systems theory)}.
Internat. Ser. Numer. Math. \textbf{91}, Birkh\"{a}user 1989.

\bibitem{CZ95}
R.F.~Curtain, H.J.~Zwart.
\emph{An Introduction to Infinite-Dimensional Linear Systems Theory}.
Springer-Verlag, 1995.

\bibitem{DU:77}
J.~Diestel, and J.J.~Uhl.
\emph{Vector Measures},
American Mathematical Society, 1977.

\bibitem{Tanja} T.~Eisner.
\emph{Stability of Operators and Operators Semigroup}.
Birkh\"auser, 2010.

\bibitem{maximal} K.-J.~Engel, M.~Kramar Fijavz, B.~Kl\"oss, R.~Nagel, E.~Sikolya.
\emph{Maximal controllability for boundary control problems}.
App. Math. Optim \textbf{62} (2010), 205--227

\bibitem{ST} K.-J.~Engel.
\emph{Spectral theory and generator property for one-sided coupled operator matrices}.
Semigroup Forum \textbf{58} (1999), 267--295.

\bibitem{adm} K.-J.~Engel.
\emph{On the characterization of admissible control- and observation operators}.
Systems \& Control Letters \textbf{34} (1998), 225--227.

\bibitem{EN} K.-J.~Engel, R.~Nagel.
\emph{One-Parameter Semigroups for Linear Evolution Equations}.
Springer-Verlag, 1999.

\bibitem{GC} P.~Grabowski, F.M.~Callier.
\emph{Admissible observation operators. Semigroup criteria of admissibility}.
Integral Equations Operator Theory \textbf{25} (1996), 182--198.

\bibitem{greiner} G.~Greiner.
\emph{Perturbing the boundary conditions of a generator}.
Houst. J. Math. \textbf{13} (1987), 213--229.

\bibitem{LP} P.D.~Lax, R.S.~Phillips.
\emph{Scattering Theory}.
Academic Press, 1967.

\bibitem{Haase} M.~Haase.
\emph{The Functional Calculus for Sectorial Operators, Operator Theory, Advances and Applications.}
Birkh\"auser, 2006.

\bibitem{Hadd1} S.~Hadd, A.~Idrissi.
\emph{Regular linear systems governed by systems with state, input and output delays.}
IMA J. Math. Control Inform. \textbf{22} (2005), 423--439.

\bibitem{Hadd2} S.~Hadd, A.~Idrissi, A.~Rhandi.
\emph{The regular linear systems associated with the shift semigroups and application to control linear systems with delay.}
Math. Control Signal Systems \textbf{18} (2006), 272--291.

\bibitem{Hel:76} J.W.~Helton.
\emph{Systems with infinite-dimensional state space: the Hilbert space approach}.
Proc. IEEE \textbf{64} (1976), 145--160.

\bibitem{rusho} L.F.~Ho, D.L.~Russel.
\emph{Admissible input elements for systems in Hilbert spaces and a Carleson measure criterion}.
SIAM J. Control Optim. \textbf{21} (1983), 614--640.

\bibitem{JP:04} B.~Jacob, J.R.~ Partington.
\emph{Admissibility of control and observation operators for semigroups: a survey}. In: {Current Trends in Operator Theory and its Applications}, {Oper. Theory Adv. Appl.} \textbf{149}, pp.~{199--221}, {Birkh\"auser} {2004}.

\bibitem{LR:05} Y.~Latushkin, F.~R\"{a}biger
\emph{Operator Valued Fourier Multipliers and Stability of Strongly Continuous Semigroups}.
Integr. Equ. Oper. Theory \textbf{51} (2005), 375--394.

\bibitem{Nei:81} H.~Neidhardt.
\emph{On abstract linear evolution equations, I}.
Math. Nachr. \textbf{103} (1981), 283--298.

\bibitem{Sta:05} O.~Staffans.
\emph{Well-posed Linear Systems}.
Encyclopedia of Mathematics and its Applications \textbf{103}, Cambridge University Press, 2005.

\bibitem{tesi2} O.~Staffans, G.~Weiss.
\emph{Transfer functions of regular linear systems. Part II: The system operator and the Lax-Phillips Semigroup}.
Trans. Amer. Math. Soc. \textbf{354} (2002), 3229--3262.

\bibitem{nuovoweiss} O.~Staffans, G.~Weiss.
\emph{A physically motivated class of scattering passive linear systems}.
SIAM J. Control Optim. \textbf{50} (2012), 3083--3112.

\bibitem{stein} E.M.~Stein.
\emph{Singular Integrals and Differentiability Properties of Functions}.
Princeton University Press, 1970

\bibitem{Weiss} M.~Tucsnak, G.~Weiss.
\emph{Observation and Control for Operator Semigroups}.
Birkh\"auser, 2009.

\bibitem{diagon} G.~Weiss.
\emph{Admissibility of input element for diagonal semigroups on $\ell^2$.}
Systems \& Control Letters \textbf{10} (1988), 79--82.

\bibitem{W-Ad-contr} G.~Weiss.
\emph{Admissibility of unbounded control operators}.
SIAM J. Control Optim. \textbf{27} (1989), 527--545.

\bibitem{lebext} G.~Weiss.
\emph{Admissible observation operators for linear semigroups}.
Israel J. Math. \textbf{65} (1989), 17--43.

\bibitem{reghilbert} G.~Weiss.
\emph{The representation of regular linear systems on Hilbert spaces}.
Internat. Ser. Numer. Math. \textbf{91}, Birkh\"{a}user 1989.

\bibitem{Wei:91} G.~Weiss.
\emph{Two conjectures on the admissibility of control operators}. In: F.~Kappel,
W.~Desch (eds.), Estimation and Control of Distributed Parameter Systems, pp.~367--378, Birkh\"{a}user 1991.

\bibitem{carl} G.~Weiss.
\emph{The operator Carlson measure criterion for admissibility of control operators for diagonal semigroups on $\ell^2$}.
Systems \& Control Letters \textbf{16} (1991), 219--227.

\bibitem{tesi1} G.~Weiss.
\emph{Transfer functions of regular linear systems. Part I: Characterizations of regularity}.
Trans. Amer. Math. Soc. \textbf{342} (1994), 827--854.

\bibitem{Wei:98} G.~Weiss.
\emph{A powerful generalization of the Carleson measure theorem}. In: V.~Blondel,
E.~Sontag, M.~Vidyasagar, and J.~Willems (eds.), Open Problems in Mathematical
Systems Theory and Control, Springer Verlag, 1998.

\end{thebibliography}
\end{document}